\documentclass{amsart}
\usepackage{amssymb,amsmath,latexsym,times,xcolor,hyperref,tikz}

\hypersetup{colorlinks=true, linkcolor=blue, citecolor=blue}

\numberwithin{equation}{section}

\theoremstyle{plain}
\newtheorem{thm}{Theorem}[section]

\newtheorem{cor}[thm]{Corollary}

\newtheorem{lemma}[thm]{Lemma}

\newtheorem*{lem*}{Lemma}
\newtheorem*{thm*}{Theorem}
\newtheorem*{cor*}{Corollary}

\theoremstyle{definition}
\newtheorem{example}{Example}

\newcommand{\meet}{\land}
\newcommand{\join}{\lor}
\DeclareMathOperator{\cl}{cl}

\newcommand{\frp}{\mathbin{\Box}}

\title[Free products and extremal matroids]{Free products, extremal
  matroids, and a generalization of perfect matroid designs}
\author[J.~Bonin]{Joseph E.~Bonin} \address
{Department of Mathematics\\ The George Washington University\\
  Washington, D.C.\ 20052, USA} \email {jbonin@gwu.edu} \date{\today}


\begin{document}

\begin{abstract}
  Extremal matroids are those that yield vertices in the polytope of
  matroids.  We show how to get the point, in the appropriate polytope
  of matroids, for the free product $M\frp M'$ of matroids $M$ and
  $M'$ from the points for $M$ and $M'$ in their polytopes of
  matroids.  With this we show that if $M\frp M'$ is extremal, then
  $M$ and $M'$ are extremal.  The converse is false.  We identify a
  large class of matroids that includes perfect matroid designs and
  sparse paving matroids, and we find sufficient conditions under
  which free products of extremal matroids in this class are extremal.
\end{abstract}

\maketitle

\section{Introduction}\label{sec:intro}

For a matroid $M$, its \emph{cyclic flats} are the flats $F$ that are
unions of circuits; equivalently, the restriction $M|F$ has no
coloops.  The lattice $\mathcal{Z}(M)$ of cyclic flats of $M$, ordered
by inclusion, plays many important roles.  For instance, the structure
of the chains in $\mathcal{Z}(M)$ that include the least and greatest
cyclic flats of $M$ determines all valuative invariants of $M$.  A
matroid $N$ is \emph{nested} (also called a Schubert matroid) if
$\mathcal{Z}(N)$ is a chain. Derksen and Fink \cite{DerksenFink}
showed that, in terms of the indicator functions of matroid base
polytopes, the matroid base polytope of any matroid can be uniquely
written as a linear combination, with integer coefficients, of the
matroid base polytopes of nested matroids of the same rank on the same
set.

The indicator function of the matroid base polytope is not a matroid
invariant: it is different on different matroids, even if they are
isomorphic.  The importance of matroid invariants motivates studying
the symmetrized indicator function $\mathbb{I}(\mathcal{P}(M))$ of the
matroid base polytope $\mathcal{P}(M)$ of a matroid $M$.  We can apply
this matroid invariant to a matroid or to an unlabeled matroid (i.e.,
an isomorphism class of matroids).  Let $N_1,N_2,\ldots,N_k$ be the
unlabeled rank-$r$ nested matroids on $n$ elements.  Derksen and
Fink's result extends to $\mathbb{I}(\mathcal{P}(M))$: for any
(labeled or unlabeled) rank-$r$ matroid $M$ on $n$ elements, there is
a unique vector $\boldsymbol{v}_M=(a_1,a_2,\ldots,a_k)$ of integers
for which
$$\mathbb{I}(\mathcal{P}(M))=
a_1\,\mathbb{I}(\mathcal{P}(N_1))+
a_2\,\mathbb{I}(\mathcal{P}(N_2))+\cdots
+a_k\,\mathbb{I}(\mathcal{P}(N_k)).
$$
The \emph{polytope $\Omega_{r,n}$ of unlabeled matroids}, which
Ferroni and Fink \cite{LuisAlex} introduced, is the convex hull of the
set of vectors $\boldsymbol{v}_M$ as $M$ ranges over all rank-$r$
(labeled or unlabeled) matroids on $n$ elements.  They showed that
$\boldsymbol{v}_M$ is a vertex of $\Omega_{r,n}$ if and only if the
set of matroids $M'$ for which $\boldsymbol{v}_{M'} =\boldsymbol{v}_M$
can be characterized by maximizing a sequence of valuative invariants
(see Theorem \ref{thm:vertbyinv}).  Such matroids are called
\emph{extremal matroids}.

Let $M$ and $M'$ be matroids on disjoint ground sets.  Ferroni and
Fink \cite{LuisAlex} showed that if the direct sum $M\oplus M'$ is
extremal, then $M$ and $M'$ are extremal, and they conjectured that the
converse holds.  The direct sum $M\oplus M'$ is one of many matroids
$K$ on $E(M)\cup E(M')$ for which $K|E(M)=M$ and $K/E(M)=M'$.  Among
such matroids, $M\oplus M'$ is least in the weak order.  (Recall that
for matroids $M_1$ and $M_2$ on a set $E$, we have $M_1\leq M_2$ in
the \emph{weak order}, or $M_2$ is \emph{freer} than $M_1$, if all
independent sets of $M_1$ are independent in $M_2$.)  There is also a
freest such matroid $K$: it is the \emph{free product} $M\frp M'$ of
$M$ and $M'$, which was introduced by Crapo and Schmitt
\cite{fpm}. The independent sets of $M\frp M'$ are the sets $I\cup J$
where $I$ is an independent set of $M$, $J\subseteq E(M')$, and
$|I\cup J|\leq r(M)+r_{M'}(J)$.
  
The fact that free products of nested matroids are nested makes it
natural to explore the interplay between the free product and the
polytope of matroids, and especially extremal matroids.  That is the
theme of this paper.

If we know the expressions for the symmetrized indicator functions of
the matroid base polytopes of $M$ and $M'$ in terms of those of nested
matroids, expressing the counterpart for the direct sum $M\oplus M'$
is fairly complicated.  In contrast, it is strikingly simple for the
free product $M\frp M'$.  In our first main result, Theorem
\ref{thm:freeproductindicator}, we show that if
$$\mathbb{I}(\mathcal{P}(M))=\sum_{i\in S}a_i\,
\mathbb{I}(\mathcal{P}(N_i)) \qquad \text{ and } \qquad
\mathbb{I}(\mathcal{P}(M'))=\sum_{j\in T}b_j\,
\mathbb{I}(\mathcal{P}(N'_j))$$ then
$$\mathbb{I}(\mathcal{P}(M\frp M'))=\sum_{i\in S,j\in T}
a_ib_j\, \mathbb{I}(\mathcal{P}(N_i\frp N'_j)),$$ where $N_i$ and
$N'_j$ range over the unlabeled nested matroids whose ranks and sizes
match $M$ and $M'$, respectively.

Theorem \ref{thm:freeproductindicator} is the key to our second main
result, Theorem \ref{thm:freeextremalcomponents}: if $M\frp M'$ is
extremal, then so are $M$ and $M'$.  This is a counterpart of the
result of Ferroni and Fink that we mentioned above.

The counterpart, for free products, of Ferroni and Fink's conjecture
for direct sums is false by the following observation.  Uniform
matroids are extremal, and iterated free products of uniform matroids
yield all nested matroids, but while some nested matroids are
extremal, not all are.  Thus, free products of extremal matroids can
be, but are not always, extremal.  In Section \ref{sec:SD}, we
identify sufficient conditions under which free products of extremal
matroids are extremal.  These results concern a large class of
matroids, semi-designs, that we introduce and to which we turn.

Semi-designs properly generalize perfect matroid designs (matroids in
which all flats of the same rank have the same size).  For a matroid
$M$, let $R_{\mathcal{Z}(M)}$ (or $R_{\mathcal{Z}}$ if there is no
ambiguity) be $\{r(A)\,:\,A\in\mathcal{Z}(M)\}$.  Instead of requiring
that all flats of the same rank have the same size (as in perfect
matroid designs), property (SD1) below imposes this size requirement
only on cyclic flats.  Property (SD2), which says that cyclic flats of
larger rank have larger nullity, allows, but does not require,
containment among any cyclic flats of different ranks.  We say that a
matroid $M$ is a \emph{semi-design} if there are integers $c_i$, for
each $i\in R_{\mathcal{Z}}$, such that
\begin{itemize}
\item[(SD1)] $|A|=c_{r(A)}$ for all $A\in\mathcal{Z}(M)$ and
\item[(SD2)] if $i,j\in R_{\mathcal{Z}}$ and $i>j$, then
  $c_i-i> c_j-j$.
\end{itemize}
Perfect matroid designs are semi-designs, as are all sparse paving
matroids.  Free products and duals of semi-designs are semi-designs.
The third main result of the paper, Theorem \ref{thm:frpbeyondPMD},
gives sufficient conditions for free products of extremal semi-designs
to be extremal.  Theorem \ref{thm:frpbeyondPMD} is applied to free
products of affine and projective planes in Corollary \ref{cor:pps}
and to nested matroids in Corollary \ref{cor:nested}.  In the other
direction, Theorem \ref{thm:notext} shows that a nested matroid is not
extremal if it has simultaneous violations of the inequalities in
Corollary \ref{cor:nested}.  Section \ref{sec:examples} treats three
infinite families of extremal semi-designs that satisfy the conditions
in Theorem \ref{thm:frpbeyondPMD} and are not derived from perfect
matroid designs.

The necessary background is reviewed in the following section.

\section{Background}\label{sec:background}

For general matroid background, we refer to \cite{oxley}; we follow
the notation in that source.  Recall that for a matroid $M$ and set
$X\subseteq E(M)$, the \emph{nullity} $\eta(X)$ of $X$ is $|X|-r(X)$.
The nullity $\eta(M)$ of $M$ is $\eta(E(M))$.  From the submodular
inequality for the rank function, we immediately get the supermodular
inequality for nullity, that is, for all $X,Y\subseteq E(M)$,
$$\eta(X\cup Y)+\eta(X\cap Y)\geq \eta(X)+\eta(Y).$$
If $\eta(M)>0$, the \emph{girth} $g(M)$ of $M$ is the smallest size of
a circuit of $M$.  If $r(M)>0$, the \emph{cogirth} $g(M^*)$ of $M$ is
the smallest size of a cocircuit of $M$.

A matroid is labeled: the ground set is crucial, as is, for instance,
knowing which of its subsets are independent.  When we care about
matroids only up to isomorphism, we may call an isomorphism class of
matroids an \emph{unlabeled matroid}.

We will use the M\"obius function $\mu:L\times L\to \mathbb{Z}$ of a
lattice $L$.  Knowing the recursive definition suffices: (i)
$\mu(a,b)=0$ if $a\not\leq b$, (ii) $\mu(a,a)=1$ for all $a$, and
(iii) if $a<b$, then $\mu(a,b)=-\sum_{x\,:a<x\leq b}\mu(x,b)$.  Those
desiring more information might consult, e.g., Rota \cite{rota} or
Stanley \cite{ec1}.

We use $\mathbb{N}$ for the set of positive integers, $[n]$ for the
set $\{1,2,\ldots,n\}$, and $[0,n]$ for $[n]\cup\{0\}$.

\subsection{Cyclic flats}\label{subsec:cycfl}

As stated in Section \ref{sec:intro}, a \emph{cyclic flat} of a
matroid $M$ is a flat $F$ that is a (possibly empty) union of
circuits, or, equivalently, $M|F$ has no coloops.  The set
$\mathcal{Z}(M)$ of cyclic flats of $M$ is a lattice under inclusion:
for $A,B\in \mathcal{Z}(M)$, the join $A\join B$ is $\cl(A\cup B)$ and
the meet $A\meet B$ is $(A\cap B)-U$ where $U$ is the set of coloops
of $M|(A\cap B)$.  It is straightforward to show that, for all
$X\subseteq E(M)$,
\begin{equation}\label{eq:cftorank}
  r(X)=\min\{r(F)+|X-F|\,:\,F\in\mathcal{Z}(M)\},
\end{equation}
so $M$ is determined by its ground set $E(M)$ along with the cyclic
flats of $M$ and their ranks, that is, by the set
$\{E(M)\}\cup\{(F,r(F))\,:\,F\in \mathcal{Z}(M)\}$.  We will use the
following result from \cite{sims, cycflats}, which characterizes
matroids from the perspective of cyclic flats and their ranks.

\begin{thm}\label{thm:axioms}
  For a collection $\mathcal{Z}$ of subsets of a set $E$ and a
  function $r:\mathcal{Z}\to \mathbb{Z}$, there is a matroid $M$ on
  $E$ with $\mathcal{Z}(M)=\mathcal{Z}$ and $r_M(X) =r(X)$ for all
  $X\in\mathcal{Z}$ if and only if
  \begin{itemize}
  \item[(Z0)] $(\mathcal{Z},\subseteq)$ is a lattice,
  \item[(Z1)] $r(0_{\mathcal{Z}})=0$, where $0_{\mathcal{Z}}$ is the
    least set in $\mathcal{Z}$,
  \item[(Z2)] $0<r(Y)-r(X)<|Y-X|$ for all sets $X,Y$ in $\mathcal{Z}$
    with $X\subsetneq Y$, and
  \item[(Z3)] for all pairs of sets $X,Y$ in $\mathcal{Z}$ (or,
    equivalently, just incomparable sets in $\mathcal{Z}$),
    $$r(X\join Y) + r(X\meet Y) + |(X\cap Y) - (X\meet Y)|\leq
    r(X)+r(Y).$$
  \end{itemize}
\end{thm}

From Equation (\ref{eq:cftorank}) we see that, for matroids $M$ and
$N$, a function $\phi:E(M)\to E(N)$ is an isomorphism of $M$ onto $N$
if and only if $\phi$ is a bijection that induces a bijection from
$\mathcal{Z}(M)$ onto $\mathcal{Z}(N)$, and $r_M(A)=r_N(\phi(A))$ for
all $A\in\mathcal{Z}(M)$.

Remarks at the end of \cite[Section 3]{cycflats} point the way to
using Theorem \ref{thm:axioms} to generalize the basic operation of
circuit-hyperplane relaxation.  The next result develops this further.
To be clear, in this result, for $\mathcal{Z}$ to be a sublattice of
$\mathcal{Z}(M)$, the meet and join operations must be the same:
$X\meet_{\mathcal{Z}} Y=X\meet_{\mathcal{Z}(M)} Y$ and
$X\join_{\mathcal{Z}} Y=X\join_{\mathcal{Z}(M)} Y$ for all
$X,Y\in \mathcal{Z}$.

\begin{cor}\label{cor:generalrelax}
  For a matroid $M$, let $\mathcal{Z}$ be a sublattice of
  $\mathcal{Z}(M)$ that includes the least cyclic flat of $M$.  Then
  the pair $(\mathcal{Z},r)$, where $r(A)=r_M(A)$ for all
  $A\in\mathcal{Z}$, satisfies properties \emph{(Z0)--(Z3)} in Theorem
  \ref{thm:axioms} and so defines a matroid on $E(M)$.
\end{cor}

The next lemma follows from the fact that a set is a union of circuits
in a matroid $M$ if and only if its complement is an intersection of
hyperplanes in the dual matroid $M^*$.

\begin{lemma}\label{lem:cfdual}
  For a matroid $M$, we have
  $\mathcal{Z}(M^*) =\{E(M)-A\,:\,A\in\mathcal{Z}(M)\}$.
\end{lemma}

Thus, $\mathcal{Z}(M^*)$ is isomorphic to the order dual of
$\mathcal{Z}(M)$.  The next lemma is well known and easy to prove.

\begin{lemma}\label{lem:cyclicflatsminors}
  Let $X$ and $Y$ be cyclic flats of a matroid $M$ for which
  $X\subset Y$. A subset $F$ of $Y-X$ is a cyclic flat of the minor
  $M|Y/X$ if and only if $F\cup X$ is a cyclic flat of $M$.
\end{lemma}

\subsection{Principal extensions and transversal matroids}\label{subsec:petran}
Principal extension is the operation that is used to extend a matroid
by adding a point freely to the flat that is spanned by a particular
set.  For a matroid $M$, a subset $X$ of $E(M)$, and an element
$e\notin E(M)$, define $r':2^{E(M)\cup \{e\}}\to \mathbb{Z}$ by, for
all $Y\subseteq E(M)$, setting $r'(Y)=r_M(Y)$ and
$$r'(Y\cup \{e\})=
\begin{cases}
  r_M(Y), & \text{if } X\subseteq\cl_M(Y),\\
  r_M(Y)+1, & \text{otherwise.}
\end{cases}$$
It is easy to check that $r'$ is the rank function of a
matroid on $E(M)\cup \{e\}$.  This matroid, which is denoted $M+_X e$,  is
the \emph{principal extension} of $M$ in which $e$ has been
\emph{added freely} to $X$.  Observe that $M+_X e = M+_{\cl(X)} e$ and
$\cl(X)\cup \{e\}$ is a flat of $M+_X e$.  Also, for $X,Y\subseteq E(M)$
and $e,f\not\in E(M)$, we have $(M+_X e)+_Y f=(M+_Y f)+_X e$, so the
order in which we apply a sequence of principal extensions to subsets of
$E(M)$ does not matter.  The \emph{free extension} of $M$ is
$M+_{E(M)} e$.

A \emph{set system} on a set $E$ is an indexed family
$\mathcal{A}=(A_1,A_2,\ldots,A_r)$ of (not necessarily different)
subsets of $E$.  A \emph{partial transversal} of $\mathcal{A}$ is a
subset $I$ of $E$ for which there is an injection
$\phi:I\rightarrow [r]$ with $x\in A_{\phi(x)}$ for all $x\in I$.  As
Edmonds and Fulkerson~\cite{ef} proved, the partial transversals of a
set system $\mathcal{A}$ on $E$ are the independent sets of a matroid
on $E$; we call $\mathcal{A}$ a \emph{presentation} of this
\emph{transversal matroid} $M[\mathcal{A}]$.  A transversal matroid is
\emph{fundamental} (or \emph{principal}) if for some presentation
$(A_1,A_2,\ldots,A_r)$ and each $i\in [r]$, some element in $A_i$ is
in no $A_j$ with $j\in [r]-\{i\}$.  The lemma below is well known.

\begin{lemma}\label{lem:transcfspan}
  If $F$ is a cyclic flat of $M[(A_1,A_2,\ldots,A_r)]$, then
  $r(F)=|\{i\,:\,F\cap A_i\ne \emptyset\}|$.
\end{lemma}

Brylawski \cite{aff} gave a useful way to view a transversal matroid
$M$ given a presentation $(A_1,A_2,\ldots,A_r)$ of $M$.  Let the set
$V=\{v_1,v_2,\ldots,v_r\}$ be disjoint from $E(M)$.  View the free
matroid $U_{r,r}$ on $V$ as having the elements of $V$ put at the
vertices of an $r$-vertex simplex via a bijection.  For each
$e\in E(M)$, by principal extension, put $e$ freely in the face of the
simplex that has the vertex set $\{v_i\,:\,e\in A_i\}$.  After placing
all elements of $E(M)$, delete $V$; the result is a geometric
representation of $M$.  By Lemma \ref{lem:transcfspan}, any cyclic
flat $F$ of $M$ spans a face of the simplex with $r(F)$ vertices.  A
fundamental transversal matroid can be represented in this way so
that, for each vertex of the simplex, at least one element of the
matroid is placed there.  It follows that a matroid $M$ is a
fundamental transversal matroid if and only if it has a basis $B$ for
which, for every cyclic flat $F$ of $M$, the set $F\cap B$ is a basis
of $F$.

\subsection{The free product and nested matroids}\label{subsec:freenested}

Recall from Section \ref{sec:intro} that, for matroids $M$ and $N$ on
disjoint sets, their \emph{free product} $M\frp N$ is the matroid on
$E(M)\cup E(N)$ in which the independent sets are the sets $I\cup J$
where $I$ is an independent set of $M$, $J\subseteq E(N)$, and
$|I\cup J|\leq r(M)+r_N(J)$.  Two examples are shown in Figure
\ref{fig:frpex}.  This operation was introduced by Crapo and Schmitt
\cite{fpm}.  In \cite{freeUniqFact}, they gave a variety of equivalent
formulations of the free product as well as many important properties
of this operation.  Two key properties are that this operation is
associative and that $(M\frp N)^*=N^*\frp M^*$.  This operation is not
commutative.  Two free products are well known: $M\frp U_{0,1}$ is the
free extension $M+_{E(M)}e$ of $M$; the dual operation, free
coextension of $M$, is $U_{1,1}\frp M$.  Thus, with the associative
property, when forming $M\frp N$, there is no loss of generality in
assuming that $M$ has no coloops and $N$ has no loops.  For our
purposes, the most relevant formulation of the free product is in
terms of cyclic flats: if $M$ has no coloops and $N$ has no loops,
then $M\frp N$ is the matroid on $E(M)\cup E(N)$ for which
$$\mathcal{Z}(M\frp N) = \mathcal{Z}(M)\cup \{E(M)\cup
A\,:\,A\in\mathcal{Z}(N)\}$$ and $r_{M\frp N}(A) = r_M(A)$ for all
$A\in\mathcal{Z}(M)$ while $r_{M\frp N}(E(M)\cup A) = r(M)+r_N(A)$ for
all $A\in\mathcal{Z}(N)$.  From either formulation, one easily checks
that $(M\frp N)|E(M)=M$ and $(M\frp N)/E(M)=N$.

All cyclic flats of $M\frp N$ are either subsets or supersets of
$E(M)$.  We will use \cite[Theorem 6.3]{freeUniqFact} which we state
next and which treats the converse. In the cases of interest to us,
the set $X$ below will be a cyclic flat.

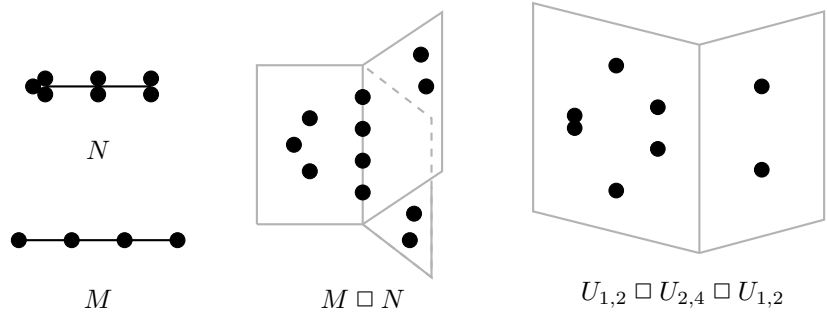
\begin{figure}
  \centering
  \begin{tikzpicture}[scale=0.7]
    \draw[thick] (-6.5,-0.3)--(-3.5,-0.3);%

    \filldraw (-6.5,-0.3) circle (4pt);%
    \filldraw (-5.5,-0.3) circle (4pt);%
    \filldraw (-4.5,-0.3) circle (4pt);%
    \filldraw (-3.5,-0.3) circle (4pt);%

    \node at (-5,-1.4) {$M$};%

     \draw[thick] (-6,2.6)--(-4,2.6);%

     \filldraw (-6,2.75) circle (4pt);%
     \filldraw (-6,2.45) circle (4pt);%
     \filldraw (-6.23,2.6) circle (4pt);%
     
     \filldraw (-5,2.75) circle (4pt);%
     \filldraw (-5,2.45) circle (4pt);%
     
     \filldraw (-4,2.75) circle (4pt);%
     \filldraw (-4,2.45) circle (4pt);%

    \node at (-5,1.4) {$N$};%

    \draw[thick, black!30] (-2,0)--(0,0)--(1.5,1)--(1.5,4)--(0,3)
    --(-2,3)--(-2,0);%
    \draw[thick, black!30] (0,0)--(0,3);%

    \draw[thick, black!30, dashed] (0,0)--(1.3,-1)--(1.3,2)--(0,3);%
    \draw[thick, black!30] (0,0)--(1.3,-1)--(1.3,0.7);%

    \filldraw (0,0.6) circle (4pt);%
    \filldraw (0,1.2) circle (4pt);%
    \filldraw (0,1.8) circle (4pt);%
    \filldraw (0,2.4) circle (4pt);%
    
    \filldraw (-1,1) circle (4pt);%
    \filldraw (-1,2) circle (4pt);%
    \filldraw (-1.3,1.5) circle (4pt);%

    \filldraw (1.2,2.6) circle (4pt);%
    \filldraw (1.1,3.2) circle (4pt);%
   
    \filldraw (0.97,0.2) circle (4pt);%
    \filldraw (0.89,-0.3) circle (4pt);%

   \node at (0,-1.4) {$M\frp N$};%
  \end{tikzpicture}
  \hspace{1cm}
  \begin{tikzpicture}[scale=0.55]
    \draw[thick,
    black!30](16,-0.5)--(19,0.3)--(19,5.3)--(16,4.5)--(12,5.5)
    --(12,0.5)--(16,-0.5);%
    \draw[thick, black!30](16,-0.5)--(16,4.5);%

    \filldraw (17.5,1.5) circle (5pt);%
    \filldraw (17.5,3.5) circle (5pt);%
    \filldraw (13,2.5) circle (5pt);%
    \filldraw (13,2.8) circle (5pt);%
    \filldraw (14,1) circle (5pt);%
    \filldraw (14,4) circle (5pt);%
    \filldraw (15,3) circle (5pt);%
    \filldraw (15,2) circle (5pt);%
    \node at (15.6,-1.5) {$U_{1,2}\frp U_{2,4}\frp U_{1,2}$};%
  \end{tikzpicture}
  \caption{The free product $M\frp N$ of the rank-$2$ matroids $M$ and
    $N$ has rank $4$.  The nested matroid
    $U_{1,2}\frp U_{2,4}\frp U_{1,2}$ also has rank $4$.}
  \label{fig:frpex}
\end{figure}

\begin{thm}\label{thm:pinchpoint}
  For a matroid $K$, if $X\subseteq E(K)$ and, for each cyclic flat
  $F$ of $K$, either $X\subseteq F$ or $F\subseteq X$, then
  $K=(K|X)\frp (K/X)$.
\end{thm}

We will also use the following result about the uniqueness of
factorizations into free products \cite[Theorem 6.18]{freeUniqFact}.

\begin{thm}\label{thm:frpuniquefactorization}
  Assume that $M_1$ and $M_2$ are matroids on disjoint sets, as are
  $N_1$ and $N_2$, and that $|E(M_1)|=|E(N_1)|$.  If $M_1\frp M_2$ is
  isomorphic to $N_1\frp N_2$, then $M_i$ is isomorphic to $N_i$ for
  each $i\in[2]$.
\end{thm}

Note that nested matroids (i.e., matroids $M$ for which
$\mathcal{Z}(M)$ is a chain) are precisely the iterated free products
of uniform matroids.  For unlabeled nested matroids, all that matters
in the free product
$U_{r_1,n_1}\frp U_{r_2,n_2}\frp\cdots\frp U_{r_k,n_k}$ is the rank,
$r_i$, and nullity, $n_i-r_i$, of each $U_{r_i,n_i}$.  As noted above,
we can assume that no $U_{r_i,n_i}$ except perhaps $U_{r_1,n_1}$ has
loops and that none except perhaps $U_{r_k,n_k}$ have coloops.
Unlabeled nested matroids correspond bijectively to lattice paths: map
$U_{r_1,n_1}\frp U_{r_2,n_2}\frp\cdots\frp U_{r_k,n_k}$ to the lattice
path $N^{r_1}E^{n_1-r_1}N^{r_2}E^{n_2-r_2}\ldots N^{r_k}E^{n_k-r_k}$,
where $N^a$ denotes $a$ consecutive north steps and $E^b$ denotes $b$
consecutive east steps, with each step of unit length.  For instance,
the nested matroid in Figure \ref{fig:frpex} maps to the path
$NEN^2E^2NE$.  From basic lattice path counting, it follows that the
number of unlabeled rank-$r$ nested matroids on $n$ elements is
$\binom{n}{r}$.

\subsection{Matroid base polytopes, valuative invariants, and the
  $\mathcal{G}$-invariant}\label{ssec:VIandG}

For a matroid $M$ on $[n]$ and basis $B$ of $M$, the
\emph{characteristic vector} $\mathbf{v}_B$ of $B$ is the vector in
$\mathbb{R}^n$ in which entry $i$ is $1$ if and only if $i\in B$,
otherwise entry $i$ is $0$.  The convex hull of the vectors
$\mathbf{v}_B$, as $B$ ranges over all bases of $M$, is the
\emph{matroid base polytope} $\mathcal{P}(M)$.  One can recover the
matroid $M$ from the polytope $\mathcal{P}(M)$ since the vertices of
$\mathcal{P}(M)$ give the bases of $M$.

For $X\subseteq\mathbb{R}^n$, its \emph{indicator function}
$\bar{\mathbb{I}}(X):\mathbb{R}^n\to \mathbb{R}$ is defined by
$$\bar{\mathbb{I}}(X)(x)=
\begin{cases}
  1, & \text{ if } x\in X,\\
  0, & \text{ if } x\not\in X.
\end{cases}$$
We call the indicator function $\bar{\mathbb{I}}(\mathcal{P}(M))$ of
the matroid base polytope of $M$ the  \emph{indicator function of}
$M$.  The next result is due to Derksen and Fink \cite[Theorem
5.4]{DerksenFink}. 

\begin{thm}\label{thm:dfnested}
  Let $M$ be a rank-$r$ matroid on $[n]$ and let $N_1,N_2,\ldots,N_t$
  be the nested matroids of rank $r$ on $[n]$.  There are unique
  integers $a_1,a_2,\ldots,a_t$ for which
  \begin{equation}\label{eq:lcnested}
    \bar{\mathbb{I}}(\mathcal{P}(M))=a_1\,
    \bar{\mathbb{I}}(\mathcal{P}(N_1))+ a_2\,
    \bar{\mathbb{I}}(\mathcal{P}(N_2))+\cdots+a_t\,
    \bar{\mathbb{I}}(\mathcal{P}(N_t)).
  \end{equation}
\end{thm}

Using this notation, if $a_i\ne 0$, then the set of cyclic flats of
the nested matroid $N_i$ is a chain $C$ in $\mathcal{Z}(M)$ that
includes the least and greatest cyclic flats of $M$; furthermore,
$r_{N_i}(X)=r_M(X)$ for all $X\in C$.  Thus, if $a_i\ne 0$, the nested
matroid $N_i$ has the same loops and coloops as $M$.  This justifies
an assumption that we will make when it simplifies the exposition,
namely, that $M$ has neither loops nor coloops.  For a nested matroid
$N_i$ for which $a_i\ne 0$, Hampe \cite{Hampe} gave an elegant and
useful expression for the coefficient $a_i$ using the M\"obius
function.  Let $\mathcal{C}(\mathcal{Z}(M))$ consist of all chains in
$\mathcal{Z}(M)$ that include the least and greatest cyclic flats,
ordered by inclusion, along with a greatest element $\hat{1}$.  The
intersection of two chains in $\mathcal{C}(\mathcal{Z}(M))$ is in
$\mathcal{C}(\mathcal{Z}(M))$, and $\hat{1}$ is the greatest element,
so $\mathcal{C}(\mathcal{Z}(M))$ is a lattice.  The coefficient $a_i$
of $\bar{\mathbb{I}}(\mathcal{P}(N_i))$, where $\mathcal{Z}(N_i)$ is
the chain $C\in \mathcal{C}(\mathcal{Z}(M))$, is the M\"obius value
$-\mu(C,\hat{1})$ in $\mathcal{C}(\mathcal{Z}(M))$.  Figure
\ref{fig:MuExample} gives an example.

Equation (\ref{eq:lcnested}) motivates one of several equivalent
definitions of a valuation.  First, a \emph{matroid invariant} is a
function that is defined on the set of matroids and that assigns the
same value to isomorphic matroids.  A matroid invariant $f$ taking
values in an abelian group $A$ is \emph{valuative}, or is a
\emph{valuation}, if, for any matroids $M_1,M_2,\ldots,M_t$ of rank
$r$ on $[n]$ and any integers $a_1,a_2,\ldots,a_t$, if
$$a_1\,
\bar{\mathbb{I}}(\mathcal{P}(M_1))(x)+ a_2\,
\bar{\mathbb{I}}(\mathcal{P}(M_2))(x)+\cdots+a_t\,
\bar{\mathbb{I}}(\mathcal{P}(M_t))(x)=0$$ for all $x\in\mathbb{R}^n$,
then $a_1\, f(M_1)+a_2\, f(M_2)+\cdots+a_t\, f(M_t)$ is the identity
of $A$.  Thus, by Theorem \ref{thm:dfnested}, the values of a
valuative invariant $f$ on (unlabeled) nested matroids determine $f$.
For example, for the matroid $M\frp N$ in Figure \ref{fig:frpex} and a
valuative invariant $f$, we have
$$f(M\frp N) = f(U_{2,4}\frp U_{1,3}\frp U_{1,4}) +2
f(U_{2,4}\frp U_{1,2}\frp U_{1,5}) - 2 f(U_{2,4}\frp U_{2,7}) $$
where, using the notation for the cyclic flats established in Figure
\ref{fig:MuExample}, the first term on the right is from the chain
$\{\emptyset, A, X,E\}$ (writing $E$ for $E(M\frp N)$), the next term
is from the two chains $\{\emptyset, A, Y,E\}$ and
$\{\emptyset, A, Z,E\}$, and the last term is from
$\{\emptyset, A, E\}$, which has M\"obius value $2$.

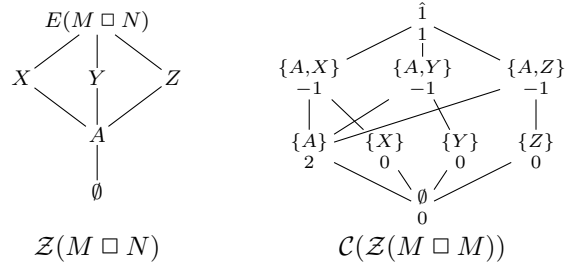
\begin{figure}
  \centering
  \hspace{6pt}
  \begin{tikzpicture}[scale=1]
    \node[inner sep = 0.3mm] (em) at (1,0) {\footnotesize
      $\emptyset$};%
    \node[inner sep = 0.3mm] (a) at (1,0.75) {\footnotesize $A$};%
    \node[inner sep = 0.3mm] (b) at (0,1.5) {\footnotesize $X$};%
    \node[inner sep = 0.3mm] (c) at (1,1.5) {\footnotesize $Y$};%
    \node[inner sep = 0.3mm] (d) at (2,1.5) {\footnotesize $Z$};%
    \node[inner sep = 0.3mm] (e) at (1,2.25) {\footnotesize
      $E(M\frp N)$};%

    \foreach \from/\to in {em/a,a/b,a/c,a/d,b/e, c/e,d/e}
    \draw(\from)--(\to);%
    
    \node at (1,-0.75) {$\mathcal{Z}(M\frp N)$};%
  \end{tikzpicture}
  \hspace{1cm}
  \begin{tikzpicture}[scale=1]
    \node[inner sep = 0.3mm] (em) at (-0.5,-0.2) {
       $\genfrac{}{}{0pt}{}{\emptyset}{0}$};%

    \node[inner sep = 0.3mm] (a) at (-2,0.55) {
      $\genfrac{}{}{0pt}{}{\{A\}}{2}$};%
    \node[inner sep = 0.3mm] (b) at (-1,0.55) {
      $\genfrac{}{}{0pt}{}{\{X\}}{0}$};%
    \node[inner sep = 0.3mm] (c) at (0,0.55) {
      $\genfrac{}{}{0pt}{}{\{Y\}}{0}$};%
    \node[inner sep = 0.3mm] (d) at (1,0.55) {
      $\genfrac{}{}{0pt}{}{\{Z\}}{0}$};%

    \node[inner sep = 0.3mm] (ab) at (-2,1.5) {
      $\genfrac{}{}{0pt}{}{\{A,X\}}{-1}$} ;%
    \node[inner sep = 0.3mm] (ac) at (-0.5,1.5) {
      $\genfrac{}{}{0pt}{}{\{A,Y\}}{-1}$};%
    \node[inner sep = 0.3mm] (ad) at (1,1.5) {
      $\genfrac{}{}{0pt}{}{\{A,Z\}}{-1}$ };%
    
    \node[inner sep = 0.3mm] (t) at (-0.5,2.25) {
      $\genfrac{}{}{0pt}{}{\hat{1}}{1}$};%

    \foreach \from/\to in {em/a,em/b,em/c,em/d,
      a/ab,a/ac,a/ad,b/ab,c/ac,
      d/ad,ab/t,ac/t,ad/t} \draw(\from)--(\to);%

    \node at (-0.5,-0.75) {$\mathcal{C}(\mathcal{Z}(M\frp M))$};%
\end{tikzpicture}  
\caption{On the left, the lattice of cyclic flats of the matroid
  $M\frp N$ in Figure \ref{fig:frpex}; on the right, its lattice of
  chains of cyclic flats.  The flat $A$ has rank $2$, size $4$; the
  flat $X$ has rank $3$, size $7$; flats $Y$ and $Z$ have rank $3$,
  size $6$.  The cyclic flats $\emptyset$ and $E(M\frp N)$ are in each
  chain but are not listed to reduce clutter.  The number below a
  chain $C$ is $\mu(C,\hat{1})$. }\label{fig:MuExample}
\end{figure}

The next result, from \cite[Theorem 5.6]{catdata}, gives the valuative
invariants that we will use.

\begin{lemma}\label{lem:listvalinv}
  Each of the following is a valuative invariant of a matroid $M$:
  \begin{enumerate}
  \item $F_{h,k}(M;s_h,s_{h+1},\ldots,s_k)$, which, for any
    $h,k,s_h,s_{h+1},\ldots,s_k\in \mathbb{N}\cup\{0\}$ for which
    $0\leq h\leq k\leq r(M)$, is the number of flags
    $F_h\subsetneq F_{h+1}\subsetneq \cdots\subsetneq F_k$ of flats of
    $M$ that satisfy $r(F_i)=i$ and $ |F_i|=s_i$ for all $i$ for which
    $h\leq i\leq k$, and
  \item $f_k(M;s,c)$, which, for any $k,s,c\in \mathbb{N}\cup\{0\}$,
    is the number of rank-$k$ size-$s$ flats $X$ of $M$ for which
    $M|X$ has $c$ coloops.
  \end{enumerate}
\end{lemma}

Setting $h=k$ in item (1) gives the number of flats of $M$ that have a
specified rank and size.  The number of cyclic flats of $M$ of a
specified rank and size is a valuative invariant; set $c=0$ in item
(2).  Linear combinations of real-valued valuative invariants are
valuative, so, for instance, the total number of cyclic flats of a
matroid is a valuation.  For a matroid $M$ and a set $S$ of $r(M)+1$
nonnegative integers $s_0<s_1<\cdots<s_{r(M)}$, we shorten
$F_{0,r(M)}(M;s_0,s_1,\ldots,s_{r(M)})$ to $F(M;S)$.

We turn to a valuative invariant that was introduced by Derksen
\cite{derksen} and that Derksen and Fink \cite{DerksenFink} proved to
be universal among valuative invariants.  For a rank-$r$ matroid $M$
and permutation $\pi=e_1e_2\ldots e_n$ of $E(M)$, the \emph{rank
  sequence} $\underline{r}(\pi)$ of $\pi$ is the sequence
$r_1 r_2 \ldots r_n$ where $r_1 = r(\{e_1\})$ and for $j \geq 2$,
$$r_j = r(\{e_1, e_2,\ldots,e_j\}) - r(\{e_1, e_2,\ldots,e_{j-1}\}).$$
Thus, $r$ entries of $\underline{r}(\pi)$ are $1$ and $n-r$ are $0$;
we call such sequences \emph{$(n,r)$-sequences}.  Fix integers $n$ and
$r$ with $0\leq r\leq n$.  For each $(n,r)$-sequence $\underline{r}$,
let $[\underline{r}]$ be a formal symbol.  Let $\mathcal{G}(n,r)$ be
the vector space, over a field $F$ of characteristic zero, that
consists of all formal linear combinations, with coefficients in $F$,
of the symbols $[\underline{r}]$.  For a matroid $M$ of rank $r$ on
$n$ elements, its \emph{$\mathcal{G}$-invariant} is defined by
\begin{equation}\label{eq:GDef}
  \mathcal{G}(M) = \sum_{\pi} [\underline{r}(\pi)]
\end{equation}
where the sum is over all $n!$ permutations $\pi$ of $E(M)$.  The
universality result of Derksen and Fink \cite{DerksenFink} says that
one obtains any valuative invariant $f$ with values in an abelian
group $A$ by assigning values in $A$ to each symbol $[\underline{r}]$
and extending linearly, that is, getting $f(M)$ by replacing each
symbol $[\underline{r}(\pi)]$ in Equation (\ref{eq:GDef}) by its image
$f([\underline{r}(\pi)])$ in $A$.

We will use the following result from \cite[Proposition
4.11]{catdata}.

\begin{lemma}\label{lem:Gfrpr}
  For matroids $M$ and $N$ on disjoint ground sets,
  $\mathcal{G}(M\frp N)$ can be calculated from $\mathcal{G}(M)$ and
  $\mathcal{G}(N)$.
\end{lemma}

A matroid $M$ is \emph{$\mathcal{G}$-unique} if all matroids $N$ for
which $\mathcal{G}(N)=\mathcal{G}(M)$ are isomorphic to $M$.  Theorems
\ref{thm:ex2}, \ref{thm:gen4swirl}, and \ref{thm:bic} each show that
certain semi-designs are both extremal and $\mathcal{G}$-unique.

\subsection{The polytope of unlabeled matroids and extremal
  matroids}\label{subsec:polytope}

To get a counterpart, for an unlabeled matroid, of the matroid base
polytope, Ferroni and Fink \cite{LuisAlex} used the \emph{symmetrized
  indicator function} $\mathbb{I}(X):\mathbb{R}^n\to \mathbb{R}$ of a
subset $X$ of $\mathbb{R}^n$, which is given by
$$
\mathbb{I}(X)(x_1,x_2,\ldots,x_n)=\frac{1}{n!}\sum_{\sigma\in S_n}
\bar{\mathbb{I}}(X)(x_{\sigma(1)},x_{\sigma(2)},\ldots,x_{\sigma(n)})$$
where $S_n$ is the symmetric group on $[n]$.  (The bar in the notation
distinguishes between the indicator function $\bar{\mathbb{I}}$, as on
the right above, and the symmetrized indicator function $\mathbb{I}$,
as on the left.)  If $M$ and $N$ are different but isomorphic matroids
on $[n]$, then we have
$\bar{\mathbb{I}}(\mathcal{P}(M))\ne \bar{\mathbb{I}}(\mathcal{P}(N))$
while $\mathbb{I}(\mathcal{P}(M))=\mathbb{I}(\mathcal{P}(N))$; thus,
$\mathbb{I}(\mathcal{P}(M))$, unlike
$\bar{\mathbb{I}}(\mathcal{P}(M))$, is a matroid invariant.  This
justifies letting the symmetrized indicator function of an unlabeled
matroid be the symmetrized indicator function of any member of its
isomorphism class.  Fix a linear order $N_1,N_2,\ldots,N_{\binom nr}$
on the $\binom{n}{r}$ unlabeled rank-$r$ nested matroids on $[n]$.
Theorem \ref{thm:dfnested} carries over to unlabeled matroids: for any
(labeled or unlabeled) rank-$r$ matroid $M$ on $[n]$, there is a
unique vector $\boldsymbol{v}_M=(a_1,a_2,\ldots,a_{\binom nr})$ of
integers for which
\begin{equation}\label{eq:symindfcnform}
  \mathbb{I}(\mathcal{P}(M))=\sum_{i\,:\,1\leq i\leq \binom nr}
  a_i\,\mathbb{I}(\mathcal{P}(N_i)).
\end{equation}
The convex hull of the set of vectors $\boldsymbol{v}_M$, over all
rank-$r$ matroids $M$ on $[n]$, is the \emph{polytope $\Omega_{r,n}$
  of unlabeled rank-$r$ matroids on $[n]$}.  A matroid $M$ for which
$\boldsymbol{v}_M$ is a vertex of $\Omega_{r,n}$ is called
\emph{extremal}.  Extremal matroids are characterized in the following
result by Ferroni and Fink \cite[Lemma 4.3]{LuisAlex}.

\begin{thm}\label{thm:vertbyinv}
  Fix integers $r$ and $n$ with $0\leq r\leq n$.  Let
  $f_1,f_2,\ldots,f_t$ be real-valued valuative invariants.  Let $S_0$
  be the set of all unlabeled rank-$r$ matroids on $[n]$, and let
  $$S_i=\{M\in S_{i-1}\,:\, f_i(M)\geq f_i(N) \text{ for all } N\in
  S_{i-1}\}$$ for each $i\in [t]$.  If
  $\boldsymbol{v}_M=\boldsymbol{v}_N$ for all $M,N\in S_t$, then
  $\boldsymbol{v}_M$ is a vertex of $\Omega_{r,n}$, so all matroids in
  $S_t$ are extremal.
\end{thm}

Since $f$ is valuative if and only if $-f$ is, we can also minimize
real-valued valuative invariants.

For example, \cite[Theorem 5.2]{maxvalinv} shows that the cycle
matroid $M(K_{r+1})$ of the complete graph $K_{r+1}$ is extremal since
it is the only matroid that satisfies the following sequence of
optimization criteria, cast in the notation of Lemma
\ref{lem:listvalinv}: among rank-$r$ matroids $M$ on $\binom{r+1}{2}$
elements, maximize $F_{1,1}(M;1)$, then minimize $F_{2,2}(M;j)$ for
all $j\not\in\{2,3\}$, then minimize $F_{3,3}(M;j)$ for all
$j\not\in\{3,4,6\}$, then, for all $h$ with $4\leq h<r$, minimize
$F_{h,h}(M;j)$ for all $j>\binom{h+1}{2}$, and lastly maximize
$F_{2,2}(M;3)$.  In words: focus on simple matroids, restrict the
sizes of lines to $2$ and $3$, restrict the sizes of planes to $3$,
$4$, and $6$, and ban rank-$h$ flats that have more than
$\binom{h+1}{2}$ elements; among the matroids that remain,
$M(K_{r+1})$ is the only one that maximizes the number of $3$-point
lines.  Another example is given in Lemma \ref{lem:PMD} below.

A result that is implicit in \cite{LuisAlex} and is worth noting is
that for matroids $M$ and $N$, we have
$\boldsymbol{v}_M=\boldsymbol{v}_N$ if and only if
$\mathcal{G}(M)=\mathcal{G}(N)$, so whether $M$ is extremal is
determined by $\mathcal{G}(M)$.  Thus, all matroids that give a
particular point in $\Omega_{r,n}$ are isomorphic if and only if they
are $\mathcal{G}$-unique.  As noted at the beginning of \cite[Section
4.3]{LuisAlex}, a matroid $M$ is extremal for $\Omega_{r,n}$ if and
only if $M^*$ is extremal for $\Omega_{n-r,n}$.

\subsection{Perfect matroid designs}\label{subsec:PMD}
A \emph{perfect matroid design} is a matroid $M$ in which the size of
a flat depends only on its rank, that is, there are integers
$\alpha_0,\alpha_1,\ldots, \alpha_{r(M)}$ for which
$|F|=\alpha_{r(F)}$ for all flats $F$ of $M$.  (See \cite{deza,pmd}.)
We call $\alpha_0,\alpha_1,\ldots, \alpha_{r(M)}$ the \emph{size
  parameters} of $M$.  Examples of perfect matroid designs include
(finite) projective and affine geometries and their truncations,
uniform matroids, and paving matroids of rank $r$ on $[n]$ for which
the hyperplanes are the blocks of a Steiner system $S(r-1,k,n)$ (i.e.,
a set of $k$-subsets (blocks) of $[n]$ for which each $(r-1)$-subset
of $[n]$ is a subset of exactly one block).  As shown in \cite[Example
3.4]{catdata}, if $M$ and $N$ are perfect matroid designs with the
same size parameters, then $\mathcal{G}(M)=\mathcal{G}(N)$.  The
result below is \cite[Theorem 3.1]{maxvalinv}.

\begin{lemma}\label{lem:PMD}
  Perfect matroid designs are extremal.  If
  $S=\{\alpha_0,\alpha_1,\ldots, \alpha_r\}$, where
  $\alpha_{i-1}<\alpha_i$ for $i\in[r]$, is the set of size parameters
  of some perfect matroid design, then the matroids that maximize the
  valuative invariant $F(N;S)$ on rank-$r$ matroids $N$ on $\alpha_r$
  elements are the perfect matroid designs that have the set $S$ of
  size parameters.
\end{lemma}

For instance, projective planes of order $q$ are extremal: they are
the rank-$3$ matroids on $q^2+q+1$ elements that maximize
$F(M;\{0,1,q+1,q^2+q+1\})$.  

\section{The indicator function of the free product of
  matroids}\label{sec:IndFuncFRP}

The first main result, Theorem \ref{thm:freeproductindicator}, shows
how to write $\mathbb{I}(\mathcal{P}(M\frp M'))$ in the form of
equation (\ref{eq:symindfcnform}) from the corresponding expressions
for $\mathbb{I}(\mathcal{P}(M))$ and $\mathbb{I}(\mathcal{P}(M'))$.

\begin{thm}\label{thm:freeproductindicator}
  Assume that the ground sets of the matroids $M$ and $M'$ are
  disjoint and that $M$ has no coloops and $M'$ has no loops.  Let
  $N_1,N_2,\ldots,N_s$ be the unlabeled nested matroids of rank $r(M)$
  on $|E(M)|$ elements, and let $N'_1,N'_2,\ldots,N'_t$ be the
  unlabeled nested matroids of rank $r(M')$ on $|E(M')|$ elements.  If
  $$\mathbb{I}(\mathcal{P}(M))=\sum_{i\in[s]}a_i\,
  \mathbb{I}(\mathcal{P}(N_i)) \qquad \text{ and } \qquad
  \mathbb{I}(\mathcal{P}(M'))=\sum_{j\in[t]}b_j\,
  \mathbb{I}(\mathcal{P}(N'_j))$$ 
  then
  $$\mathbb{I}(\mathcal{P}(M\frp M'))=\sum_{i\in[s],j\in[t]}
  a_ib_j\, \mathbb{I}(\mathcal{P}(N_i\frp N'_j)).$$ The corresponding
  statement for the indicator functions of
  matroids 
  also holds.
\end{thm}

As noted earlier, free products of nested matroids are nested, so each
$N_i\frp N'_j$ is nested.  To prepare for the proof of Theorem
\ref{thm:freeproductindicator}, we treat a binary operation on
lattices.  Bennett \cite{MK} defined and studied the order-dual of
this operation, calling it the rectangular product of lattices.  Given
lattices $L_1$ and $L_2$ with greatest elements $\hat{1}_1$ and
$\hat{1}_2$ respectively, let $L_i^\circ=L_i-\{\hat{1}_i\}$, for
$i\in[2]$.  Let
$L= (L_1^\circ\times L_2^\circ)\cup\{(\hat{1}_1,\hat{1}_2)\}$.  The
\emph{d-product} of $L_1$ and $L_2$ is the lattice on $L$ that is
induced by the direct product, that is, for $(a,b),(c,d)\in L$, we
have $(a,b)\leq_L(c,d)$ if and only if $a\leq_{L_1} c$ and
$b\leq_{L_2} d$.  It is easy to see that $L$ is a lattice.  In the
next lemma, we identify the values
$\mu_L((a,b),(\hat{1}_1,\hat{1}_2))$ of the M\"obius function of $L$
for $(a,b)\in L_1^\circ\times L_2^\circ$.  As noted in \cite{MK}, the
order-dual, the rectangular product, arises naturally in various
contexts; for instance, the face lattice of the product
$P_1\times P_2$ of two polytopes is the rectangular product of the
face lattices of $P_1$ and $P_2$.  The lemma below therefore may be
known, and it can surely be seen through different lenses (e.g., Euler
characteristics).  We give a simple proof to make this paper
self-contained.

\begin{lemma}\label{lem:dproduct}
  Let $L_1$ and $L_2$ be lattices, let $L$ be their d-product, and let
  $\mu_1$, $\mu_2$, and $\mu_L$ be the corresponding M\"obius
  functions.  If $(a,b)\in L_1^\circ\times L_2^\circ$, then
  $$\mu_L((a,b),(\hat{1}_1,\hat{1}_2)) =
  -\mu_1(a,\hat{1}_1)\mu_2(b,\hat{1}_2).$$ 
\end{lemma}

\begin{proof}
  Define $f:L\to\mathbb{Z}$ by setting $f((\hat{1}_1,\hat{1}_2))=1$
  and, for $(a,b)\in L_1^\circ\times L_2^\circ$, setting
  $f((a,b)) = -\mu_1(a,\hat{1}_1)\mu_2(b,\hat{1}_2)$.  The lemma
  follows from the recursive definition of the M\"obius function by
  showing that the sum of the values of $f$ over any non-singleton
  interval $[(a,b), (\hat{1}_1,\hat{1}_2)]$ of $L$ is $0$.  That holds
  since, for any $(a,b)\in L_1^\circ\times L_2^\circ $, we have
  \begin{align*}
    f((\hat{1}_1,\hat{1}_2)) + \sum\limits_{\substack{(c,d)\in
    L_1^\circ\times L_2^\circ \,:\\ (a,b)\leq_L(c,d)}}
    f((c,d))
    & = 1 + \sum_{c \in L_1^\circ \,:\, a\leq_{L_1}c} \mu_1(c,\hat{1}_1)
      \sum_{d\in L_2^\circ  \,:\,b\leq_{L_2} d} -\mu_2(d,\hat{1}_2) \\
    & = 1 + \sum_{c \in L_1^\circ \,:\, a\leq_{L_1}c} \mu_1(c,\hat{1}_1)
      \,\mu_2(\hat{1}_2,\hat{1}_2) \\
    & =\mu_1(\hat{1}_1,\hat{1}_1)   + \sum_{c \in L_1^\circ \,:\, a\leq_{L_1}c}
      \mu_1(c,\hat{1}_1)\\
    & = 0. \qedhere
  \end{align*}
  \end{proof}

\begin{proof}[Proof of Theorem \ref{thm:freeproductindicator}]
  It suffices to prove the result for the indicator functions of
  matroids since the result for symmetrized indicator functions
  follows from that by the distributive property.  Thus, the
  coefficients of interest are M\"obius values of chains in
  $\mathcal{C}(\mathcal{Z}(M))$, $\mathcal{C}(\mathcal{Z}(M'))$, and
  $\mathcal{C}(\mathcal{Z}(M\frp M'))$.  Let $S=E(M)$ and $T=E(M')$.
  As noted after Theorem \ref{thm:dfnested}, without loss of
  generality we may assume that $M$ has no loops and $M'$ has no
  coloops.

  First consider the chains $C$ in
  $\mathcal{C}(\mathcal{Z}(M\frp M'))$ that include $S$. Such chains
  have the form
  $$\emptyset\subset F_1\subset\cdots \subset F_n\subset
  S\subset S \cup G_1\subset \cdots \subset S\cup G_{n'}\subset S\cup
  T$$ where $C_M=\{\emptyset, F_1,\ldots, F_n, S\}$ is any chain in
  $\mathcal{C}(\mathcal{Z}(M))$ and
  $C_{M'}=\{\emptyset, G_1,\ldots , G_{n'}, T\}$ is any chain in
  $\mathcal{C}(\mathcal{Z}(M'))$.  Thus, the interval $[C,\hat{1}]$ in
  $\mathcal{C}(\mathcal{Z}(M\frp M'))$ is isomorphic to the
  $d$-product of the intervals $[C_M,\hat{1}]$ in
  $\mathcal{C}(\mathcal{Z}(M))$ and $[C_{M'},\hat{1}]$ in
  $\mathcal{C}(\mathcal{Z}(M'))$.  Lemma \ref{lem:dproduct} gives
  $-\mu_{M\frp M'}(C,\hat{1}) = \mu_M(C_M,\hat{1}) \cdot
  \mu_{M'}(C_{M'},\hat{1})$, where the subscripts indicate to which of
  the lattices $\mathcal{C}(\mathcal{Z}(M\frp M'))$,
  $\mathcal{C}(\mathcal{Z}(M))$, or $\mathcal{C}(\mathcal{Z}(M'))$ a
  M\"obius function refers.  Also, the nested matroid that $C$ yields
  is the free product of the nested matroid for $C_M$ with the nested
  matroid for $C_{M'}$.  From this, the result follows once we show
  that $\mu_{M\frp M'}(C,\hat{1})=0$ for all chains
  $C\in\mathcal{C}(\mathcal{Z}(M\frp M'))$ for which $S\not\in C$.
  
  For a chain $C\in\mathcal{C}(\mathcal{Z}(M\frp M'))$ for which
  $S\not\in C$, let
  $$k(C)=\max\{|C'-C|\,:\,C' \text{ is a chain in }
  \mathcal{Z}(M\frp M') \text{ and }C\subseteq C'\}.$$ To prove that
  $\mu_{M\frp M'}(C,\hat{1})=0$, we induct on $k(C)$.  If $k(C)=1$,
  then $C\cup\{S\}$ is the only chain $X$ in
  $\mathcal{C}(\mathcal{Z}(M\frp M'))$ for which $C\subsetneq X$.  The
  recursive definition of the M\"obius function and the equality
  $\mu_{M\frp M'}(C\cup\{S\},\hat{1})+\mu_{M\frp
    M'}(\hat{1},\hat{1})=0$ give $\mu_{M\frp M'}(C,\hat{1})=0$.  Now
  assume that $k(C)>1$ and that $\mu_{M\frp M'}(C',\hat{1})=0$ for all
  $C'\in\mathcal{C}(\mathcal{Z}(M\frp M'))$ for which $S\not\in C'$
  and $k(C')<k(C)$.  The defining recurrence for the M\"obius
  function, applied to the interval $[C\cup\{S\},\hat{1}]$ in
  $\mathcal{C}(\mathcal{Z}(M\frp M'))$, gives
  $$ 0=\mu_{M\frp M'}(\hat{1},\hat{1})+
  \sum\limits_{\substack{C'\in\mathcal{C}(\mathcal{Z}(M\frp M'))\,:\\
      C\subsetneq C', S\in C}}\mu_{M\frp M'}(C',\hat{1}).$$ Omitting
  the condition $S\in C$ just adds terms that, by the induction
  hypothesis, are zero, so
  $$0=\mu_{M\frp M'}(\hat{1},\hat{1})+\sum\limits_{\substack{C'\in
      \mathcal{C}(\mathcal{Z}(M\frp M'))\,:\\ C\subsetneq
      C'}}\mu_{M\frp M'}(C',\hat{1}).$$ By the recurrence for the
  M\"obius function, the right side of this equation is
  $-\mu_{M\frp M'}(C,\hat{1})$, so $\mu_{M\frp M'}(C,\hat{1})=0$, as
  needed.
\end{proof}

By Theorem \ref{thm:frpuniquefactorization}, the only free product
factorization of the unlabeled nested matroid $N_i\frp N'_j$ in
Theorem \ref{thm:freeproductindicator} into an unlabeled matroid on
$|E(M)|$ elements and one on $|E(M')|$ elements is the factorization
into $N_i$ and $N'_j$.  With that, the corollary below follows
immediately from Theorem \ref{thm:freeproductindicator}.

\begin{cor}\label{cor:uniqueproduct}
  Let $M_1$ and $M_2$ be matroids on ground sets that are disjoint
  from that of the matroid $M$.  Assume that $M_1$ and $M_2$ have no
  coloops and that $M$ has no loops.  If
  $\mathbb{I}(\mathcal{P}(M_1\frp M))= \mathbb{I}(\mathcal{P}(M_2\frp
  M))$, then
  $\mathbb{I}(\mathcal{P}(M_1))= \mathbb{I}(\mathcal{P}(M_2))$.
  Equivalently, if $\boldsymbol{v}_{M_1}\ne \boldsymbol{v}_{M_2}$,
  then $\boldsymbol{v}_{M_1\frp M}\ne \boldsymbol{v}_{M_2\frp M}$.
  Likewise, if $\boldsymbol{v}_{M_1}\ne \boldsymbol{v}_{M_2}$, then
  $\boldsymbol{v}_{M\frp M_1}\ne \boldsymbol{v}_{M\frp M_2}$.
\end{cor}

We offer another perspective on Corollary \ref{cor:uniqueproduct}: we
can compute $\mathcal{G}(M)$ and $\mathcal{G}(M')$ from
$\mathcal{G}(M\frp M')$ along with any one of $r(M)$, $r(M')$,
$|E(M)|$, and $|E(M')|$.  (Recall the comments at the end of Section
\ref{subsec:polytope} connecting $\mathcal{G}(M)$ and
$\boldsymbol{v}_M$.)  We focus on the case in which $M$ has no coloops
and $M'$ has no loops; treating the general case from that is routine
(see \cite[Proposition 4.3]{catdata}).  Now $E(M)$ is the only
rank-$r(M)$ cyclic flat of $M\frp M'$ and it contains all cyclic flats
of lower rank, so $E(M)$ is the unique largest set in $M\frp M'$ of
rank $r(M)$.  The terms $[\underline{r}(\pi)]$ in
$\mathcal{G}(M\frp M')$ in which exactly $r(M)$ of the first $|E(M)|$
entries are $1$ come from the permutations $\pi$ of $E(M\frp M')$ in
which the elements of $E(M)$ are the first $|E(M)|$ entries.
Therefore taking the sum of all such terms $[\underline{r}(\pi)]$ in
$\mathcal{G}(M\frp M')$ and truncating each $[\underline{r}(\pi)]$ to
its first $|E(M)|$ entries gives $|E(M')|!\cdot \mathcal{G}(M)$, so we
get $\mathcal{G}(M)$.  Similarly we get $\mathcal{G}(M')$.  We note
that, in contrast, it seems likely to be considerably harder, if it is
even feasible, to compute $\mathcal{G}(M)$ and $\mathcal{G}(M')$ from
$\mathcal{G}(M\oplus M')$ along with $r(M)$, $r(M')$, $|E(M)|$, or
$|E(M')|$.

Theorem \ref{thm:freeextremalcomponents} is the second main result of
this section.  We formulate it for free products of two matroids, but
a routine induction argument extends the result to free products of
any number of matroids.

\begin{thm}\label{thm:freeextremalcomponents}
  Let $M$ and $M'$ be matroids on disjoint ground sets.  Assume that
  $M$ has no coloops and that $M'$ has no loops.  If $M\frp M'$ is
  extremal, then $M$ and $M'$ are extremal.
\end{thm}

\begin{proof}
  It suffices to prove that $M$ is extremal since duality preserves
  being extremal and $(M\frp M')^*=M'^*\frp M^*$.  For that, it
  suffices to prove that if $M$ is not extremal, then neither is
  $M\frp M'$.  Assume that $M$ is not extremal.  Thus, there are
  rank-$r(M)$ matroids $K_1,K_2,\ldots,K_u$ on $|E(M)|$ elements,
  where $u\geq 2$ and $\boldsymbol{v}_{K_i}\ne \boldsymbol{v}_{M}$ and
  $\boldsymbol{v}_{K_i}\ne \boldsymbol{v}_{K_j}$ for all
  $\{i,j\}\subseteq[u]$, along with positive real numbers
  $c_1,c_2,\ldots,c_u$ for which $c_1+c_2+\cdots+c_u = 1$ and
  \begin{equation}\label{eq:convexcomb}
    \mathbb{I}(\mathcal{P}(M))=\sum_{h\in[u]}c_h\,
    \mathbb{I}(\mathcal{P}(K_h)).
  \end{equation}
  Corollary \ref{cor:uniqueproduct} gives
  $\boldsymbol{v}_{K_i\frp M'}\ne \boldsymbol{v}_{M\frp M'}$ and
  $\boldsymbol{v}_{K_i\frp M'}\ne \boldsymbol{v}_{K_j\frp M'}$ for all
  $\{i,j\}\subseteq[u]$, so we can conclude that $M\frp M'$ is not
  extremal by showing the equality
  \begin{equation}\label{eq:convexcombfrp}
    \mathbb{I}(\mathcal{P}(M\frp M'))=\sum_{h\in[u]}c_h\,
    \mathbb{I}(\mathcal{P}(K_h\frp M')).
  \end{equation}  
  To carry this out, we write each of $\mathbb{I}(\mathcal{P}(M))$,
  $\mathbb{I}(\mathcal{P}(M'))$, and $\mathbb{I}(\mathcal{P}(K_h))$ as
  in equation (\ref{eq:symindfcnform}) and then use Theorem
  \ref{thm:freeproductindicator} to relate the two sides of equation
  (\ref{eq:convexcombfrp}).  Let
  $$\mathbb{I}(\mathcal{P}(M))=\sum_{i\in[s]}a_i\,
  \mathbb{I}(\mathcal{P}(N_i)), \qquad
  \mathbb{I}(\mathcal{P}(M'))=\sum_{j\in[t]}b_j\,
  \mathbb{I}(\mathcal{P}(N'_j)),$$ and
  $$\mathbb{I}(\mathcal{P}(K_h))=\sum_{i\in[s]}d_{h,i}\,
  \mathbb{I}(\mathcal{P}(N_i)),$$ where $N_1,N_2,\ldots,N_s$ are the
  unlabeled rank-$r(M)$ nested matroids on $|E(M)|$ elements, and
  $N'_1,N'_2,\ldots,N'_t$ are the unlabeled rank-$r(M')$ nested
  matroids on $|E(M')|$ elements.  Substituting the expression for
  $\mathbb{I}(\mathcal{P}(K_h))$ in the right side of equation
  (\ref{eq:convexcomb}) gives
  $$\mathbb{I}(\mathcal{P}(M))=\sum_{h\in[u], i\in[s]}c_h\,
  d_{h,i}\, \mathbb{I}(\mathcal{P}(N_i)),$$ so by the uniqueness of
  the coefficients,
  $$a_i = \sum_{h\in[u]}c_h\, d_{h,i}$$ for all $i\in [s]$.
  By Theorem \ref{thm:freeproductindicator},
  \begin{align*}
    \sum_{h\in[u]}c_h\, \mathbb{I}(\mathcal{P}(K_h\frp M'))
    & = \sum_{h\in[u]}c_h\sum_{i\in[s],j\in[t]} d_{h,i}\,b_j\,
      \mathbb{I}(\mathcal{P}(N_i\frp N'_j)) \\
    & = \sum_{j\in[t]} b_j  \sum_{i\in[s]} \sum_{h\in[u]}
      c_h\, d_{h,i}\, \mathbb{I}(\mathcal{P}(N_i\frp N'_j)) \\
     & = \sum_{j\in[t]} b_j  \sum_{i\in[s]} 
       a_i\, \mathbb{I}(\mathcal{P}(N_i\frp N'_j)) \\
    & =   \mathbb{I}(\mathcal{P}(M\frp M')),
  \end{align*}
  so $M\frp M'$ is not extremal.
\end{proof}

The converse of Theorem \ref{thm:freeextremalcomponents} is false; for
instance, $U_{1,2}$ is extremal but $U_{1,2}\frp U_{1,2}$ is not.  In
the next section, we develop some sufficient conditions under which
having $M$ and $M'$ be extremal guarantees that $M\frp M'$ is
extremal.

\section{Semi-designs}\label{sec:SD}

Recall from Section \ref{sec:intro} that, for a matroid $M$, we let
$R_{\mathcal{Z}(M)}$ (or $R_{\mathcal{Z}}$ if $M$ is the only matroid
under discussion) be $\{r(A)\,:\,A\in\mathcal{Z}(M)\}$, and that $M$
is a \emph{semi-design} if there are integers $c_i$, for each
  $i\in R_{\mathcal{Z}}$, such that
\begin{itemize}
\item[(SD1)] $|A|=c_{r(A)}$ for all $A\in\mathcal{Z}(M)$ and
\item[(SD2)] if $i,j\in R_{\mathcal{Z}}$ and $i>j$, then
  $c_i-i> c_j-j$.
\end{itemize}
This section starts with lemmas that are useful for working with
semi-designs.  The main result of this section is Theorem
\ref{thm:frpbeyondPMD}, which gives sufficient conditions for free
products of extremal semi-designs to be extremal, and which we
illustrate with its applications to free products of affine and
projective planes (Corollary \ref{cor:pps}) and to nested matroids
(Corollary \ref{cor:nested}).  While it does not address the converse
of Corollary \ref{cor:nested}, Theorem \ref{thm:notext} identifies an
option for the failure of the hypotheses in that corollary that allows
one to deduce that a nested matroid is not extremal.

\begin{lemma}\label{lem:msubiext}
  Let $M$ be a rank-$r$ semi-design and let $c_i$ be as above.  For
  each $k\in[0,r]$, let $m_k$ be the maximum size among the rank-$k$
  flats of $M$.  Fix an integer $i\in[r]$ and let
  $j=\max\{k\,:\, k\in R_{\mathcal{Z}},\, k< i\}$.
  \begin{itemize}
  \item[(1)] If $i\not\in R_{\mathcal{Z}}$, then $m_i=c_j+i-j$.  Also,
    a rank-$i$ flat has $m_i$ elements if and only if it contains a
    cyclic flat of rank $j$.
  \item[(2)] If $i\in R_{\mathcal{Z}}$, then $m_i=c_i$.  Also, a
    rank-$i$ flat that is not cyclic has at most $m_j+i-j$ elements.
  \item[(3)] We have $i\in R_{\mathcal{Z}}$ if and only if
    $m_i>m_{i-1}+1$.
  \end{itemize}
\end{lemma}

\begin{proof}
  Let $F$ be a rank-$i$ flat of $M$ that is not cyclic.  Let $U$ be
  the set of coloops of $M|F$ and let $k$ be the rank of the cyclic
  flat $F-U$.  Then $k\leq j$ and $i=r(F)=k+|U|$, so
  $$|F| =|F-U|+|U| =c_k+i-k.$$  Parts (1) and (2) follow by observing that if
  $k=j$, then $|F|= c_j+i-j$, while if $k<j$, then $c_k-k<c_j-j$, and
  so $|F| < c_j+i-j$, and in either case, if $i\in R_{\mathcal{Z}}$,
  then $|F|<c_i$.

  For part (3), if $i\in R_{\mathcal{Z}}$, then by part (2) at least
  two elements must be removed from any rank-$i$ flat of size $m_i$ to
  get a rank-$(i-1)$ flat, so $m_i>m_{i-1}+1$.  If
  $i\not\in R_{\mathcal{Z}}$, then parts (1) and (2) give
  $m_i=m_{i-1}+1$.
\end{proof}

For a semi-design $M$, we call $m_0,m_1,\ldots,m_r$, where $m_i$ is as
in Lemma \ref{lem:msubiext}, the \emph{size parameters} of $M$, and we
let $S_M$ be $\{m_0,m_1,\ldots,m_r\}$.  We always index the elements
of $S_M$ by rank, so $m_i<m_{i+1}$.  As justified by part (2) of Lemma
\ref{lem:msubiext}, below we use $m_i$ for $c_i$ in properties (SD1)
and (SD2).  By part (3), for a semi-design $M$, we can deduce
$R_{\mathcal{Z}}$ from $S_M$.  By part (1), we can deduce $S_M$ from
$R_{\mathcal{Z}}$, the set $\{m_i\,:\, i\in R_{\mathcal{Z}}\}$, and
$r(M)$.

Note that if $S_M=S_N$, then $M$ and $N$ have the same rank,
$|S_M|-1$, and their ground sets have the same size, the largest
element of $S_M$.  If we take the same definition of the set $S_M$ for
all matroids, then we can have $S_M=S_N$ even if only one of $M$ and
$N$ is a semi-design.  For instance, if $M$ is formed by taking the
parallel connection of four copies of $U_{2,4}$ at the same base point
$p$ and then deleting $p$, then $M$ is a semi-design and
$S_M=\{0,1,3,6,9,12\}$; however, if $N$ is formed by taking the
parallel connection of three copies of $U_{2,4}$ at the same base
point $p$, then deleting $p$, then taking the direct sum with
$U_{2,3}$, and finally truncating to rank five, $N$ is not a
semi-design (there are cyclic flats of rank $4$ and size $6$), but
$S_N=S_M$.

\begin{lemma}
  If $M$ is a semi-design, then so is $M^*$.
\end{lemma}

\begin{proof}
  Set $E=E(M)$, $r=r(M)$, and $S_M=\{m_0,m_1,\ldots,m_r\}$.  Recall
  Lemma \ref{lem:cfdual}.  From the equality
  $r_{M^*}(A) = |A|-r+r_M(E-A)$, the set $E-A$ has rank $i$ and size
  $m_i$ in $M$ if and only if $A$ has rank $m_r-m_i-r+i$ and size
  $m_r-m_i$ in $M^*$.  Fix $i,j\in R_{\mathcal{Z}(M)}$ for which
  $i>j$. Then $-m_i+i<-m_j+j$ since $M$ satisfies (SD2), so for cyclic
  flats of $M$ of different ranks, the cyclic flats of $M^*$ that are
  their complements have different ranks in $M^*$.  Property (SD1) for
  $M^*$ now follows.  Property (SD2) for $M^*$ requires
  that $$m_r-m_j-(m_r-m_j-r+j) > m_r-m_i -(m_r-m_i-r+i),$$ that is,
  $i>j$, so property (SD2) holds.
\end{proof}

The next lemma is immediate from the formulation of free products
using cyclic flats.

\begin{lemma}\label{lem:s1s1frp}
  Let $M$ and $N$ be matroids on disjoint ground sets.  If $M$ and $N$
  are semi-designs, then so is their free product $M\frp N$.
\end{lemma}

Like the class of perfect matroid designs, the class of semi-designs
is not closed under either minors or direct sums.  However, it is easy
to see that the following special minors are exceptions.

\begin{lemma}\label{lem:sdminor}
  If $F$ is a cyclic flat of a semi-design $M$, then both the
  restriction $M|F$ and contraction $M/F$ are semi-designs.
\end{lemma}

Recall the valuative invariant $F(M;S)$ that was defined after Lemma
\ref{lem:listvalinv}.  The next lemma relates $F(M;S_M)$, for a
semi-design $M$, to its counterpart $F(\mathcal{Z}(M))$ in
$\mathcal{Z}(M)$, where we define $F(\mathcal{Z}(M))$ to be the number
of flags of flats in $\mathcal{Z}(M)$ that include one flat of each
rank in $R_{\mathcal{Z}(M)}$.

\begin{lemma}\label{lem:pushtocyclicflats}
  Let $M$ and $N$ be rank-$r$ semi-designs with
  $S_M=S_N=\{m_0,m_1,\ldots,m_r\}$.  Then $F(M;S_M)\leq F(N;S_N)$ if
  and only if $F(\mathcal{Z}(M))\leq F(\mathcal{Z}(N))$.
\end{lemma}

\begin{proof}
  Fix consecutive ranks $j<i$ in $R_{\mathcal{Z}(M)}$ and
  $F_j, F_i\in\mathcal{Z}(M)$ for which $F_j\subset F_i$ where
  $r(F_j)=j$ and $r(F_i)=i$.  The flags
  $F_j\subset F_{j+1}\subset\cdots\subset F_i$ of $i-j+1$ distinct
  flats in $M$ correspond bijectively to the flags of $i-j+1$ distinct
  flats in the minor $M|F_i/F_j$.  Since $i$ and $j$ are consecutive
  ranks in $R_{\mathcal{Z}(M)}$, by Lemma \ref{lem:cyclicflatsminors}
  the minor $M|F_i/F_j$ is the uniform matroid $U_{i-j,m_i-m_j}$.  Now
  $F(U_{r,n};[r-1]\cup\{0,n\})$ is the falling factorial
  $(n)_{r-1} = n!/(n-r+1)!$.  Thus,
  $F(M;S_M)=d\cdot F(\mathcal{Z}(M))$ where $d$ is the product of the
  relevant falling factorials, which depend only on the numbers $i$
  and $m_i$ for all $i\in R_{\mathcal{Z}(M)}$.  Since $S_M=S_N$ and
  $R_{\mathcal{Z}(M)}=R_{\mathcal{Z}(N)}$, we get
  $F(N;S_N)=d\cdot F(\mathcal{Z}(N))$ for the same $d$, so the result
  follows.
\end{proof}

\begin{cor}\label{cor:dualmax}
  Let $M$ be a semi-design. Let $S_M=\{m_0,m_1,\ldots,m_r\}$.  Let
  $\mathcal{M}$ be the set of rank-$r$ semi-designs $N$ on $m_r$
  elements for which, if $A\in\mathcal{Z}(N)$, then
  $r_N(A)\in R_{\mathcal{Z}(M)}$ and $|A|=m_{r_N(A)}$.  Let
  $\mathcal{M}^*$ be defined analogously for the dual semi-design
  $M^*$.  The following two statements are equivalent.
  \begin{itemize}
  \item[] Among all matroids $N\in \mathcal{M}$, the matroid $M$
    maximizes $F(N;S_M)$.
  \item[] Among all matroids $N \in \mathcal{M}^*$, the matroid $M^*$
    maximizes $F(N;S_{M^*})$.
  \end{itemize}
\end{cor}

\begin{proof}
  This follows immediately since
  $F(\mathcal{Z}(N)) =F(\mathcal{Z}(N^*))$ by Lemma \ref{lem:cfdual}.
\end{proof}

Note that we get the sets $\mathcal{M}$ and $\mathcal{M}^*$ above by
optimizing valuative invariants: for $\mathcal{M}$, for each rank $i$,
if $i\not\in R_{\mathcal{Z}(M)}$, then minimize the number of rank-$i$
cyclic flats, and if $i\in R_{\mathcal{Z}(M)}$, then minimize the
number of rank-$i$ cyclic flats of sizes other than $m_i$.

We turn to the main result of this section.  The nested matroid
$U_{1,2}\frp U_{1,2}$, which is not extremal, shows that, in the
absence of assumptions beyond those made below, the inequalities in
conditions (i)--(i$''$) are optimal.
  
\begin{thm}\label{thm:frpbeyondPMD}
  Let $M_1,M_2,\ldots,M_n$, where $n\geq 2$, be semi-designs that have
  neither loops nor coloops.  Assume that at least one of the
  following three conditions holds:
  \begin{itemize}
  \item[(i)] $r(M_i)\leq g(M_{i+1})-2$ for each $i\in[n-1]$,
  \item[(i$'$)] $\eta(M_{i+1})\leq g(M^*_i)-2$ for each $i\in[n-1]$,
  \item[(i$''$)] $n\geq 3$ and there is a $k\in[n-2]$ for which, if
    $i\in[k]$, then $r(M_i)\leq g(M_{i+1})-2$, while if
    $i\in[n-1]-[k]$, then $\eta(M_{i+1})\leq g(M^*_i)-2$.
  \end{itemize}
  Assume also that
  \begin{itemize}
  \item[(ii)] for each $i\in[n]$, among all semi-designs $K$ for which
    $S_K=S_{M_i}$, the valuative invariant $F(K;S_{M_i})$ is maximal
    if and only if $\mathcal{G}(K)=\mathcal{G}(M_i)$.
  \end{itemize}
  Then the iterated free product $M_1\frp M_2\frp\cdots \frp M_n$ is
  extremal.
\end{thm}

\begin{proof}
  Let $N_i$ be $M_1\frp M_2\frp\cdots \frp M_i$ for each $i\in[n]$,
  let $N_0$ be the empty matroid, and let $N$ be $N_n$.  Each $N_i$ is
  a semi-design by Lemma \ref{lem:s1s1frp}.  Let the elements of $S_N$
  be, in increasing order, $n_0,n_1,\ldots,n_{r(N)}$.  If
  $j=r(N_{h-1})+s$ and $s\leq r(M_h)$, then Lemma \ref{lem:msubiext}
  gives $n_j=|E(N_{h-1})| +m_{h,s}$, where $m_{h,s}$ is the maximum
  size of the rank-$s$ flats of $M_h$.
  
  By minimizing the appropriate valuative invariants, we may select,
  among rank-$r(N)$ matroids on $|E(N)|$ elements, those that
  \begin{itemize}
  \item[(a)] for each $i\in[0,r(N)]$, have no rank-$i$ flats of size
    greater than $n_i$,
  \item[(b)] for each $i\in R_{\mathcal{Z}(N)}$, have no rank-$i$
    cyclic flats of size less than $n_i$,
  \item[(c)] for each $i\in[r(N)]- R_{\mathcal{Z}(N)}$, have no
    rank-$i$ cyclic flats,
  \end{itemize}
  and from those, we may then
  \begin{itemize}
  \item[(d)] select the matroids $K$ that maximize the valuative
    invariant $F(K;S_N)$, and
  \item[(e)] from those, select the matroids that minimize the total
    number of cyclic flats.
  \end{itemize}
  We claim that any matroid $K$ that is selected by items (a)--(e) is
  $M'_1\frp M'_2\frp\cdots \frp M'_n$ for some matroids
  $M'_1,M'_2,\ldots,M'_n$ for which
  $\mathcal{G}(M'_i)=\mathcal{G}(M_i)$ for each $i\in[n]$.  Since
  $\mathcal{G}( M_1\frp M_2\frp\cdots \frp M_n)= \mathcal{G}(M'_1\frp
  M'_2\frp\cdots \frp M'_n)$ by Lemma \ref{lem:Gfrpr}, this proves the
  theorem.  By selection criteria (a)--(d), the matroid $K$ is a
  semi-design and $S_N = S_K$, so
  $R_{\mathcal{Z}(N)}=R_{\mathcal{Z}(K)}$ by Lemma \ref{lem:msubiext}.

  For $i\in[0,n]$, let $s_i=r(N_i)$.  We claim that the semi-design
  $K$ has a unique flag of flats
  $F_{s_0}\subset F_{s_1}\subset F_{s_2}\subset\cdots\subset F_{s_n}$
  for which $r(F_{s_i})=s_i$ and $|F_{s_i}|=n_{s_i}$ for each
  $i\in[0,n]$.  By Lemma \ref{lem:msubiext}, all flats in such a flag
  are cyclic.  There is at least one such flag since $N$ optimizes the
  valuative invariants in items (a)--(c) and $F(N;S_N)>0$.  We first
  prove statements (1) and (2) below and then deduce uniqueness from
  them.  These statements are related by duality, but we prove both to
  make it easier to see that the subscripts (which are crucial to our
  application) line up correctly.

  \begin{itemize}
  \item[(1)] \emph{If $r(M_i)\leq g(M_{i+1})-2$, then any
      rank-$s_{i-1}$ cyclic flat of $K$ is a subset of at most one
      rank-$s_i$ cyclic flat of $K$.}
  \item[(2)] \emph{If $\eta(M_{i+1})\leq g(M^*_i)-2$, then any
      rank-$s_{i+1}$ cyclic flat of $K$ contains at most one
      rank-$s_i$ cyclic flat of $K$.}
  \end{itemize}

  Assume, contrary to statement (1), that $r(M_i)\leq g(M_{i+1})-2$,
  that $X$ is a rank-$s_{i-1}$ cyclic flat of $K$, and that $F$ and
  $F'$ are different rank-$s_i$ cyclic flats for which
  $X\subseteq F\cap F'$.  Now $\cl_K(F\cup F')$ is a cyclic flat and
  $s_i<r_K(F\cup F')$.  The submodular inequality gives
  \begin{align*}
    r_K(F\cup F')
    &\leq  r_K(F)+r_K(F')-r_K(F\cap F')\\
    &\leq  r_K(F)+r_K(F')-r_K(X)\\
    &= s_i +r(M_i)\\
    &\leq s_i +g(M_{i+1})-2.
  \end{align*}
  The least rank of a cyclic flat of $N$ greater than $s_i$ is
  $s_i+g(M_{i+1})-1$.  Since $S_K=S_N$ and $K$ is a semi-design, it
  follows that the inequality $s_i<r_K(Y)<s_i+g(M_{i+1})-1$ fails for
  all $Y\in\mathcal{Z}(K)$.  This contradiction proves statement (1).

  Before proving statement (2), we derive a key inequality.  By Lemma
  \ref{lem:msubiext}, the maximal-sized hyperplanes of $N_i$ are the
  sets $E(N_{i-1})\cup H$ where $H$ is a maximal-sized hyperplane of
  $M_i$; also, $\eta(N_i) =\eta_N(E(N_{i-1})\cup H)+g(M^*_i)-1$ for
  any such $H$.  It follows that, for any rank-$s_i$ cyclic flat $F$
  of $K$ and any flat $Y$ of $K$ for which $r_K(Y)<s_i$, we have
  \begin{equation}\label{eq:nullityeq}
    \eta_K(F) \geq \eta_K(Y)+g(M^*_i)-1.
  \end{equation}

  To prove statement (2), assume that $\eta(M_{i+1})\leq g(M^*_i)-2$,
  that $X$ is a rank-$s_{i+1}$ cyclic flat of $K$, and, contrary to
  statement (2), that $F$ and $F'$ are different rank-$s_i$ cyclic
  flats for which $ F\cup F'\subseteq X$.  Note that
  $\eta_K(F')+ \eta(M_{i+1})=\eta_K(X)$.  Since nullity is
  supermodular, we have
  \begin{align*}
    \eta_K(F\cap F')
    &\geq  \eta_K(F)+\eta_K(F')-\eta_K(F\cup F')\\
    &\geq  \eta_K(F)+\eta_K(F')-\eta_K(X)\\
    &= \eta(F)- \eta(M_{i+1}).
  \end{align*}
  Thus, $\eta_K(F) \leq \eta_K(F\cap F')+ \eta(M_{i+1})$, and so
  $\eta_K(F) \leq \eta_K(F\cap F')+ g(M^*_i)-2$ by the inequality that
  we assumed.  This contradiction to inequality (\ref{eq:nullityeq})
  proves statement (2).

  For the uniqueness of the flag
  $F_{s_0}\subset F_{s_1}\subset \cdots\subset F_{s_n}$ identified
  above under condition (i), since $F_{s_0}=\emptyset$ is unique,
  condition (i) in the case of $i=1$ gives the uniqueness of
  $F_{s_1}$, from which the uniqueness of $F_{s_2}$ follows using
  $i=2$, and so on. Likewise under condition (i$'$), since
  $F_{s_n}=E(N)$ is unique, condition (i$'$) in the case of $i=n-1$
  gives the uniqueness of $F_{s_{n-1}}$, from which we get the
  uniqueness of $F_{s_{n-2}}$ using $i=n-2$, and so on.  Finally,
  under condition (i$''$), the first argument gives the uniqueness of
  each of $F_{s_0}, F_{s_1}, \ldots, F_{s_k}$ and the second gives the
  uniqueness of each of $F_{s_{k+1}}, F_{s_{k+2}}, \ldots, F_{s_n}$,
  so the whole flag is unique.

  We use the flats $F_{s_1}, F_{s_2}, \ldots, F_{s_n}$ to define $n$
  minors of $K$: set $M'_1=K|F_{s_1}$ and, for each $i\in[n-1]$, set
  $M'_{i+1}=K|F_{s_{i+1}}/F_{s_i}$.  Since each flag that $F(K;S_N)$
  counts includes each of the flats
  $F_{s_0}, F_{s_1},\ldots, F_{s_n}$, maximizing $F(K;S_N)$ is
  equivalent to maximizing $F(M'_i;S_{M_i})$ for all $i\in [n]$.  By
  assumption (ii), for each $i\in[n]$, the invariant $F(M'_i;S_{M_i})$
  is maximized precisely when $\mathcal{G}(M'_i)=\mathcal{G}(M_i)$ for
  each $i\in [n]$.

  Let
  $\mathcal{Z} = \{X\in\mathcal{Z}(K)\,:\, F_{s_{i-1}}\subseteq X
  \subseteq F_{s_i} \text{ for some } i\in[n]\}$.  Note that
  $\mathcal{Z}$ is a sublattice of $\mathcal{Z}(K)$ and that each
  cyclic flat that is in a flag that $F(K;S_N)$ counts is in
  $\mathcal{Z}$.  By Corollary \ref{cor:generalrelax}, the pair
  $(\mathcal{Z},r)$, where $r(X)=r_N(X)$ for all $X\in\mathcal{Z}$,
  yields a matroid $K'$ on $E(K)$.  Since $F_{s_n}=E(K)$, we have
  $E(K)\in\mathcal{Z}$, so $r(K')=r(K)=r(N)$.  Clearly $K'$ is a
  semi-design, and the chains that $F(K;S_N)$ counts show that
  $S_{K'}=S_N$.  Lemma \ref{lem:pushtocyclicflats} gives
  $F(K';S_N)= F(K;S_N)$, so by selection criterion (e), we have
  $K=K'$.  Thus, $\mathcal{Z}(K)=\mathcal{Z}$, and so Theorem
  \ref{thm:pinchpoint} gives
  $$K=(K|F_{s_1})\frp (K|F_{s_2}/F_{s_1}) \frp\cdots \frp
  (K|F_{s_n}/F_{s_{n-1}}),$$ that is,
  $K=M'_1\frp M'_2\frp\cdots \frp M'_n$, as needed.
\end{proof}

We illustrate this theorem with two corollaries about free products of
perfect matroid designs.  Note that hypothesis (ii) in Theorem
\ref{thm:frpbeyondPMD} holds for all perfect matroid designs by Lemma
\ref{lem:PMD}, so only one of the conditions (i)--(i$''$) needs to be
imposed.

\begin{cor}\label{cor:pps}
  Let each of $M_1,M_2,\ldots,M_n$ be a projective or affine plane and
  let $q_i$ be the order of $M_i$. If $q_1>q_2>\cdots> q_n$, then
  $M_1\frp M_2\frp\cdots \frp M_n$ is extremal.
\end{cor}

\begin{proof}
  As noted above, it suffices to show that condition (i$'$) holds.
  The nullity $\eta(M_{i+1})$ is greatest if $M_{i+1}$ is a projective
  plane, in which case $\eta(M_{i+1})=q^2_{i+1}+q_{i+1}-2$.  The
  cogirth $g(M^*_i)$ is least if $M_i$ is an affine plane, in which
  case $g(M^*_i)=q^2_i-q_i $.  Thus, it suffices to prove that
  $q^2_{i+1}+q_{i+1}-2\leq q^2_i-q_i-2$.  This holds since
  $$q^2_{i+1}+q_{i+1}\leq (q_i-1)^2+q_i-1 = q^2_i-q_i.\qedhere$$
\end{proof}

The lattice path $N^{a_1}E^{b_1} N^{a_2}E^{b_2}\ldots N^{a_n}E^{b_n}$
represents the nested matroid in the next two results.

\begin{cor}\label{cor:nested}
  Let $M$ be the nested matroid
  $U_{a_1,a_1+b_1}\frp U_{a_2,a_2+b_2}\frp\cdots\frp U_{a_n,a_n+b_n}$
  where $a_i,b_i\in\mathbb{N}$.  The matroid $M$ is extremal if any of
  the following conditions holds:
   \begin{itemize}
   \item[(i)] $a_i<a_{i+1}$ for each $i\in[n-1]$,
   \item[(i$'$)] $b_{i+1} <b_i$ for each $i\in[n-1]$,
   \item[(i$''$)] $n\geq 3$ and there is some $k\in[n-2]$ for which, if
     $i\in[k]$, then $a_i<a_{i+1}$, while if
     $i\in[n-1]-[k]$, then $b_{i+1} <b_i$.
  \end{itemize}
\end{cor}

\begin{proof}
  This follows immediately since $U_{a_i,a_i+b_i}$ has rank $a_i$,
  girth $a_i+1$, nullity $b_i$, and cogirth $b_i+1$.
\end{proof}

For example, to see that the nested matroid
$U_{1,2}\frp U_{2,4}\frp U_{1,2}$ in Figure \ref{fig:frpex} is
extremal, use $k=1$ in condition (i$''$).

While it does not provide a converse of Corollary \ref{cor:nested},
the next result shows that if neither inequality in Corollary
\ref{cor:nested} holds for a particular $i$, then the nested matroid
is not extremal.

\begin{thm}\label{thm:notext}
  Let $M$ be the nested matroid
  $U_{a_1,a_1+b_1}\frp U_{a_2,a_2+b_2}\frp\cdots\frp U_{a_n,a_n+b_n}$.
  If $a_i\geq a_{i+1}$ and $ b_{i+1}\geq b_i$ for some $i\in[n-1]$,
  then $M$ is not extremal.
\end{thm}

\begin{proof}
  Set $\alpha_j=a_1+a_2+\cdots+a_j$ and $\beta_j=b_1+b_2+\cdots+b_j$
  for $j\in[n]$, and $\alpha_0=\beta_0=0$.  Fix an $i\in[n-1]$ for
  which $a_i\geq a_{i+1}$ and $ b_{i+1}\geq b_i$, and let $X$, $Y$,
  and $Z$ be the cyclic flats of $M$ of ranks $\alpha_{i-1}$,
  $\alpha_i$, and $\alpha_{i+1}$, respectively.  Thus, $|Y-X|=a_i+b_i$
  and $|Z-Y|=a_{i+1}+b_{i+1}$.  Let the set $Y'$ consist of (i) the
  elements of $X$, (ii) any $a_i-a_{i+1}$ elements of $Y-X$, and (iii)
  any $a_{i+1}+b_i$ elements of $Z-Y$.  Thus, $X\subset Y'\subset Z$
  and $|Y'-X|= a_i+b_i$, so $|Y'|=\alpha_i+\beta_i$.

  Let the matroid $M'$ on $E(M)$ be defined by the set of pairs of
  cyclic flats and ranks
  $$\{(A,r_M(A))\,:\,A\in\mathcal{Z}(M)\}\cup\{(Y',r_M(Y))\}.$$
  To see that $M'$ is indeed a matroid, only one instance of one
  property in Theorem \ref{thm:axioms} requires any argument: for
  property (Z3) for the pair $Y$ and $Y'$, we have
  \begin{align*}
    r(Y)+r(Y')
    & = 2(\alpha_{i-1}+a_i)\\
    & = \alpha_{i-1}+a_i+a_{i+1} + \alpha_{i-1}+a_i-a_{i+1}\\
    & = r(Z)+r(X)+|(Y\cap Y')-X|\\
    & = r(Y\join Y')+r(Y\meet Y')+|(Y\cap Y')-(Y\meet Y')|.
  \end{align*}
  Let the matroid $M''$ on $E(M)$ be defined by the set of pairs of
  cyclic flats and
  ranks $$\{(A,r_M(A))\,:\,A\in\mathcal{Z}(M)-\{Y\}\}.$$ For example,
  if $M$ is $U_{1,2}\frp U_{1,2}$, then $M'$ is
  $U_{1,2}\oplus U_{1,2}$ and $M''$ is $U_{2,4}$; also, if $M$ is
  $U_{2,4}\frp U_{1,3}$, then $M'$ is the parallel connection of two
  copies of $U_{2,4}$ and $M''$ is $U_{3,7}$.
  
  The greatest element $\hat{1}$ of the lattice
  $\mathcal{C}(\mathcal{Z}(M'))$ covers just two chains,
  $C=\mathcal{Z}(M)$ and $C'=(C-\{Y\})\cup\{Y'\}$.  Therefore only
  three chains in $\mathcal{C}(\mathcal{Z}(M'))$ have nonzero M\"obius
  values, namely, (i) $\mu(C,\hat{1})=-1$, (ii) $\mu(C',\hat{1})=-1$,
  and (iii) $\mu(C'',\hat{1})=1$ where
  $C''= C\cap C'=\mathcal{Z}(M'')$.  The nested matroids given by the
  first two chains are isomorphic to $M$, so
  $\mathbb{I}(\mathcal{P}(M'))= 2\,\mathbb{I}(\mathcal{P}(M))-
  \mathbb{I}(\mathcal{P}(M''))$.  Thus,
  $$\mathbb{I}(\mathcal{P}(M))=
  \frac{1}{2}\,\mathbb{I}(\mathcal{P}(M'))+
  \frac{1}{2}\,\mathbb{I}(\mathcal{P}(M'')),$$
  so $M$ is not extremal.
\end{proof}

\section{Examples}\label{sec:examples}

As noted earlier, perfect matroid designs satisfy hypothesis (ii) in
Theorem \ref{thm:frpbeyondPMD}.  In this section we treat three
infinite families of semi-designs that satisfy that hypothesis and are
neither perfect matroid designs nor obtained from them by duality or
free products.

Semi-designs are extremely common. For instance, any paving matroid in
which all dependent hyperplanes have the same size is a semi-design;
this includes all sparse paving matroids.  Mayhew, Newman, Welsh, and
Whittle \cite{asy} conjecture that, asymptotically, almost all
matroids are sparse paving; see \cite{BPP,MW,rudijorn} for progress
toward this conjecture.  However, a sparse paving matroid is extremal
if and only if it maximizes or minimizes the number of dependent
hyperplanes \cite[Theorem 5.11]{LuisAlex}.  The set of hyperplanes of
a rank-$r$ paving matroid on $n$ elements in which all hyperplanes
have $k$ elements gives a Steiner system $S(r-1,k,n)$.  Steiner
systems exist only for certain parameters $r,k,n$.  For most
parameters for which no Steiner system $S(r-1,k,n)$ exists, it is hard
to determine the maximum number of hyperplanes in rank-$r$ paving
matroids on $n$ elements in which all dependent hyperplanes have size
$k$.  Each of the three classes of extremal semi-designs that we treat
below, in contrast, has a structure that makes it fairly
straightforward to work with.

\begin{example}\label{ex:2}
  Ferroni and Fink \cite{LuisAlex} observed that uniform matroids are
  extremal.  Direct sums of uniform matroids are also extremal
  \cite[Theorem 8.1]{maxvalinv}.  If $M$ is the direct sum of $m$
  copies of $U_{r,n}$, where $1<r<n$ and $m>1$, then $M$ is a
  semi-design but not a perfect matroid design.  While truncations of
  semi-designs are semi-designs, truncations of extremal semi-designs
  need not be extremal; for instance,
  $U_{2,3}\oplus U_{2,3}\oplus U_{2,3}$ is an extremal semi-design,
  but its truncation to rank $3$ is not extremal.  In Theorem
  \ref{thm:ex2} we show that truncations of $M$ to ranks that are
  multiples of $r$ are extremal.
\end{example}

\begin{thm}\label{thm:ex2}
  Fix $r,n,m,s\in\mathbb{N}-\{1\}$ where $r<n$ and $s<m$.  Let $N$ be
  the truncation of the direct sum of $m$ copies of $U_{r,n}$ to rank
  $sr$.  Matroids isomorphic to $N$ are the only rank-$sr$ matroids
  $N'$ on $mn$ elements for which
  \begin{itemize}
  \item[(a)] any cyclic flat of $N'$ has rank $ir$ and size $in$ for
    some $i\in[0,s-1]$, or rank $sr$ and size $mn$, and
  \item[(b)] among matroid that satisfy condition \emph{(a)},
    $F(N';S_N)$ is maximized.
  \end{itemize}
  Thus, the semi-design $N$ is extremal and $\mathcal{G}$-unique.
\end{thm}

\begin{proof}
  Any matroid $N'$ that satisfies the criteria above is a semi-design
  and the equalities $S_{N'}=S_N$ and
  $R_{\mathcal{Z}(N')}=R_{\mathcal{Z}(N)}$ hold.  Recall Lemma
  \ref{lem:pushtocyclicflats}: maximizing $F(N';S_N)$ is equivalent to
  maximizing the invariant $F(\mathcal{Z}(N'))$ that was defined
  before that lemma.  Property ($*$) below is key to the proof that
  $N'$ is isomorphic to $N$.
  
  \begin{itemize}
  \item[($*$)] \emph{For each $i\in[s-1]$, if $X$ is a rank-$(i-1)r$
      cyclic flat of $N'$, and if $F$ and $F'$ are different rank-$ir$
      cyclic flats for which $X\subseteq F\cap F'$, then} (i)
    \emph{$X=F\cap F'$ and} (ii) \emph{if, in addition, $i\ne s-1$, then
      $F\cup F'$ is a rank-$(i+1)r$ cyclic flat of $N'$.  }
  \end{itemize}
  
  To prove this, the submodular inequality gives
  $$r(F\cup F')\leq r(F)+r(F')-r(F\cap F')\leq 2ir-(i-1)r=
  (i+1)r,$$ so $ir<r(F\cup F')\leq (i+1)r$.  Now $\cl(F\cup F')$ is a
  cyclic flat and so, by criterion (a), its rank is a multiple of $r$,
  so $r(F\cup F') = (i+1)r$.  With that, the submodular inequality
  also gives
  $$r(F\cap F')\leq r(F)+r(F')-r(F\cup F') = 2ir-(i+1)r= (i-1)r=r(X),$$
  so we get conclusion (i).  Thus, $|F\cup F'| = 2in-(i-1)n=(i+1)n$,
  which, if $i\ne s-1$, is the size of a rank-$(i+1)r$ cyclic flat, so
  conclusion (ii) follows.

  Using $i=1$ in property ($*$) shows that the rank-$r$ cyclic flats
  of $N'$ are pairwise disjoint, so there are at most $m$ of them;
  likewise, using $i=2$, for any rank-$r$ cyclic flat $X$, as $F$
  ranges over the rank-$2r$ cyclic flats that contain $X$, the sets
  $F-X$ are pairwise disjoint, so there are at most $m-1$ of them,
  etc.  It follows that
  $$F(\mathcal{Z}(N'))\leq m(m-1)\cdots (m -s+2) =
  F(\mathcal{Z}(N)).$$ Thus, selection criterion (b) forces $N'$ to
  have $m$ rank-$r$ cyclic flats, and so these cyclic flats partition
  $E(N')$.  Now by conclusion (ii) in property ($*$), for each
  $i\in[s-1]$, the union of any $i$ rank-$r$ cyclic flats is a
  rank-$ir$ cyclic flat, and the flags of such cyclic flats account
  for the equality $F(\mathcal{Z}(N'))= F(\mathcal{Z}(N))$.  To show
  that $N'$ is isomorphic to $N$, it suffices to show that there are
  no other cyclic flats of $N'$.  Let $F$ be a cyclic flat of $N'$,
  say with $r(F) = rt$, where $1<t<s$.  If $F$ were not a union of
  rank-$r$ cyclic flats, then it would contain an element $e$ for
  which the unique rank-$r$ cyclic flat $X$ that contains $e$ is not a
  subset of $F$.  By the submodular inequality, we would have
  $$rt=r(F)<r(F\cup X)\leq r(F)+r(X)-r(F\cap X)\leq rt+r-1<r(t+1),$$
  so $r(F\cup X)$ would not be a multiple of $r$.  This contradicts
  criterion (a) since $\cl(F\cup X)$ would be a cyclic flat.  Thus,
  $N'$ has, up to relabeling, the same cyclic flats as $N$, so there
  is a bijection $\phi:E(N)\to E(N')$ that induces a size- and
  rank-preserving bijection from $\mathcal{Z}(N)$ onto
  $\mathcal{Z}(N')$, so $N$ and $N'$ are isomorphic.

  Since criteria (a) and (b) can be cast as optimizing valuative
  invariants and we showed that any matroid that satisfies these
  criteria is isomorphic to the semi-design $N$, this proves that $N$
  is extremal and $\mathcal{G}$-unique.
\end{proof}

The next infinite family of extremal semi-designs generalizes the
matroid in Figure \ref{fig:whirlgeneralization}, which is the case of
$k=1$ and $t=2$ in the notation below.

\begin{example}\label{ex:whirlgeneralization}
  Fix $k,t\in\mathbb{N}$ and sets $A,B,C,D$ of size $2k+t$ for which
  $A\cap C=\emptyset=B\cap D$, while $A\cap B$, $B\cap C$, $C\cap D$,
  and $D\cap A$ are (necessarily pairwise disjoint) $k$-sets.  Let the
  cyclic flats of the rank-$4k$ matroid $M$ on $A\cup B\cup C\cup D$
  be as given in Table \ref{tab:Msizerank}.

  \begin{table}[h!]
    \begin{center}
      \begin{tabular}{|r||c|c|c|c|}
        \hline
        cyclic flat
        &$\emptyset$&$A,\, B, \, C, \, D $&$A\cup B, \,  B\cup C, \,
                                            C\cup D, \, D\cup A$&$E(M)$\\ \hline
        size  &$0$&$2k+t$&$3k+2t$&$4k+4t$\\ \hline
        rank &0&$2k$&$3k$&$4k$\\ \hline
      \end{tabular}

      \vspace{4pt}

      \caption{The cyclic flats, with their sizes and ranks, for the
        matroid $M$.}\label{tab:Msizerank}
    \end{center}
  \end{table}
  
  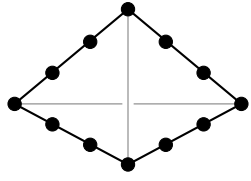
\begin{figure}
    \centering
    \begin{tikzpicture}[scale=1]
      \node[inner sep = 0mm] (1) at (0,1) {};%
      \node[inner sep = 0mm] (1') at (1.45,1) {};%
      \node[inner sep = 0mm] (2) at (1.5,0.2) {};%
      \node[inner sep = 0mm] (3) at (3,1) {};%
      \node[inner sep = 0mm] (3') at (1.55,1) {};%
      \node[inner sep = 0mm] (4) at (1.5,2.25) {};%
      \foreach \from/\to in {1/2,1/4,2/3,3/4}
      \draw[thick](\from)--(\to);%

      \foreach \from/\to in {1/1',3'/3,2/4} \draw[gray](\from)--(\to);%

      \filldraw (0,1) circle (2.5pt);%
      \filldraw (3,1) circle (2.5pt);%
      \filldraw (1.5,0.2) circle (2.5pt);%
      \filldraw (1.5,2.25) circle (2.5pt);%
      \filldraw (2,0.46) circle (2.5pt);%
      \filldraw (2.5,0.73) circle (2.5pt);%
      \filldraw (1,0.46) circle (2.5pt);%
      \filldraw (0.5,0.73) circle (2.5pt);%
      \filldraw (0.5,1.41) circle (2.5pt);%
      \filldraw (1,1.82) circle (2.5pt);%
      \filldraw (2.5,1.41) circle (2.5pt);%
      \filldraw (2,1.82) circle (2.5pt);%
    \end{tikzpicture}
    \caption{The matroid in Example \ref{ex:whirlgeneralization} in
      the case of $k=1$ and $t=2$.  The sets $A$, $B$, $C$, and $D$
      are the four $4$-point lines in this rank-$4$ matroid.}
    \label{fig:whirlgeneralization}
  \end{figure}
  
  It is easy to check that properties (Z0)--(Z3) in Theorem
  \ref{thm:axioms} hold, so $M$ is a matroid.  Four rank-$2k$ flats
  are cyclic while all others are independent, so $M$ is not a perfect
  matroid design.  The dual is a matroid of the same type with the
  roles of $k$ and $t$ switched, so it too is not a perfect matroid
  design.  Table \ref{tab:Mdualsizerank} gives the cyclic flats and
  their sizes and ranks for $M^*$.  Here, $\overline{X}$ denotes the
  complement of $X$ in $E(M)$.

  \begin{table}[h!]
    \begin{center}
      \begin{tabular}{|r||c|c|c|c|}
        \hline
        cyclic flat
        &$\emptyset$& $\overline{A\cup B}, \,  \overline{B\cup C}, \,
         \overline{C\cup D}, \, \overline{D\cup A}$&$\overline{A},\,
         \overline{B}, \, \overline{C}, \, \overline{D}
          $&$E(M)$ \rule{0pt}{11pt}\\ \hline
        size  &$0$&$2t+k$&$3t+2k$&$4t+4k$\\ \hline
        rank &0&$2t$&$3t$&$4t$\\ \hline
      \end{tabular}

      \vspace{4pt}

      \caption{The cyclic flats, with their sizes and ranks, for the
        dual $M^*$.}\label{tab:Mdualsizerank}
    \end{center}
  \end{table}

  We note that the set
  $X=(A\cap B)\cup(B\cap C)\cup(C\cap D)\cup (D\cap A)$ is a basis of
  $M$ and that each cyclic flat of $M$ is spanned by a subset of $X$,
  so $M$ is a fundamental transversal matroid and $X$ is a fundamental
  basis.

  By construction, $M$ is a semi-design.  Below we show that among all
  rank-$4k$ matroids $N$ on $4k+4t$ elements for which the size and
  rank of any cyclic flat appears in Table \ref{tab:Msizerank}, the
  unique matroid that maximizes the valuative invariant $F(N;S_M)$ is
  $M$.  From this, it follows that $M$ is both extremal and
  $\mathcal{G}$-unique.  As Lemma \ref{lem:pushtocyclicflats}
  justifies, below we focus on pairs of cyclic flats $X\subset Y$ of
  ranks $2k$ and $3k$, respectively, instead of $F(N;S_M)$.
\end{example}

\begin{thm}\label{thm:gen4swirl}
  Fix $k,t\in\mathbb{N}$.  Let $N$ be a rank-$4k$ matroid on $4k+4t$
  elements in which the size and rank of any cyclic flat is one of the
  four options given in Table \ref{tab:Msizerank}.  Then $N$ has at
  most eight flags of cyclic flats $X\subset Y$ where $r(X)= 2k$ and
  $r(Y)=3k$, and if $N$ has eight such flags, then $N$ is isomorphic
  to $M$.  Also, $M$ is extremal and $\mathcal{G}$-unique.
\end{thm}

\begin{proof}
  Let $N$ be as in the theorem.  Note that the size and rank of any
  cyclic flat of $N^*$ match one of the four options given in Table
  \ref{tab:Mdualsizerank}.  Thus, any intermediate result that we
  prove about cyclic flats of a certain rank in $N$ yields the
  counterpart of that result in $N^*$ with the roles of $k$ and $t$
  swapped, and by duality that gives a result in $N$ about cyclic
  flats of a different rank.  For example, once we show that any
  rank-$2k$ cyclic flat of $N$ is contained in at most two rank-$3k$
  cyclic flats of $N$, it follows that any rank-$2t$ cyclic flat of
  $N^*$ is contained in at most two rank-$3t$ cyclic flats of $N^*$,
  and by duality that means that any rank-$3k$ cyclic flat of $N$
  contains at most two rank-$2k$ cyclic flats of $N$.  We break the
  proof of the theorem into a sequence of structural properties of
  $N$.

  \begin{itemize}
  \item[(1)] \emph{Fix $F,F'\in\mathcal{Z}(N)$ of rank $2k$ with
      $F\ne F'$.  If $F\cap F'=\emptyset$, then $r(F\cup F')=4k$.  If
      $F\cap F'\ne \emptyset$, then $|F\cap F'|=k$ and $F\cup F'$ is a
      rank-$3k$ cyclic flat.}
  \end{itemize}

  To prove this, note that since $F,F'\in\mathcal{Z}(N)$, we have
  $\cl(F\cup F')\in\mathcal{Z}(N)$, and, since any cyclic flat has
  rank $0$, $2k$, $3k$, or $4k$, the rank of $F\cup F'$ is either $3k$
  or $4k$.  If $F\cap F'=\emptyset$, then $|F\cup F'| = 4k+2t$, which
  exceeds $3k+2t$, the size of the rank-$3k$ cyclic flats, so
  $r(F\cup F')=4k$.  Now assume that $F\cap F'\ne \emptyset$. The
  submodular inequality gives
  $$r(F\cup F')\leq r(F)+r(F') - r(F\cap F')<4k,$$ and so 
  $r(F\cup F')=3k$.  With that, the submodular inequality also gives
  $$r(F\cap F')\leq r(F)+r(F') - r(F\cup F')=2k+2k-3k=k.$$
  Now $|F\cup F'|\leq |\cl(F\cup F')| = 3k+2t$ since $\cl(F\cup F')$
  is a rank-$3k$ cyclic flat, so
  $$|F\cap F'|= |F|+|F'|-|F\cup F'| \geq 2(2k+t)-(3k+2t) = k.$$
  From the inequalities $r(F\cap F')\leq k$ and $|F\cap F'|\geq k$
  along with the observation that any set of rank less than $2k$ is
  independent in $N$, it follows that $|F\cap F'|= k$.  With that,
  inclusion/exclusion gives $|F\cup F'|= 3k+2t$, so $F\cup F'$ is the
  cyclic flat $\cl(F\cup F')$.

  \begin{itemize}
  \item[(2)] \emph{Fix distinct cyclic flats $F$, $G$, $G'$ for which
      $r(F)=2k$ and $r(G)=r(G')=3k$.  If $F\subset G$ and
      $F\subset G'$, then $F= G\cap G'$.}
  \end{itemize}
  
  For property (2), since $\cl(G\cup G')$ is a cyclic flat of rank
  greater than $3k$, we have $r(G\cup G') = 4k$.  With that, the
  submodular inequality gives
  $$r(G\cap G')\leq r(G)+r(G')-r(G\cup G') = 3k+3k-4k = 2k.$$
  Since $F\subseteq G\cap G'$ and $r(F)=2k$, we get $F=G\cap G'$.

  \begin{itemize}
  \item[(3)] \emph{A rank-$2k$ cyclic flat $F$ is a subset of at most
      two rank-$3k$ cyclic flats, and so $F$ has nonempty intersection
      with at most two other rank-$2k$ cyclic flats.}
  \end{itemize}
  
  For property (3), if $F$ were a subset of three different rank-$3k$
  cyclic flats, then, by property (2), the intersection of each pair
  of them is $F$, so their union would have size
  $2k+t+3(k+t) = 5k+4t$.  This is impossible since $|E(N)|=4k+4t$.
  Thus, the first part of property (3) holds.  The second part follows
  from that and property (1).

  As noted in the first paragraph, the next property follows from
  property (3) by duality.

  \begin{itemize}
  \item[(4)] \emph{At most two subsets of a rank-$3k$ cyclic flat are
      rank-$2k$ cyclic flats.}
  \item[(5)] \emph{Let $S$ be a set of four rank-$2k$ cyclic flats of
      $N$.  If $X\in S$, then there is a $Y\in S$ for which
      $r(X\cup Y)=4k$ and so $X\cap Y = \emptyset$.}
  \end{itemize}

  For property (5), let $S=\{X,U,V,W\}$.  Each of $\cl(X\cup U)$,
  $\cl(X\cup V)$, and $\cl(X\cup W)$ is a cyclic flat of rank either
  $3k$ or $4k$.  By property (3), at most two of them have rank
  $3k$. If, say, $r(X\cup U)=4k$, then $X\cap U = \emptyset$ by
  property (1).
  
  The upper bound on the number of flags of interest in the theorem is
  strict by property (3) if $N$ has fewer than four rank-$2k$ cyclic
  flats.  This justifies assumption (6).

  \begin{itemize}
  \item[(6)] \emph{We assume that $A$, $B$, $C$, and $D$ are rank-$2k$
      cyclic flats of $N$.}
  \item[(7)] \emph{No three rank-$2k$ cyclic flats are pairwise
      disjoint.}
  \end{itemize}

  If, contrary to property (7), the rank-$2k$ cyclic flats $A$, $B$,
  and $C$ were pairwise disjoint, then $|A\cup B\cup C| = 6k+3t$, and
  so $|\overline{A\cup B\cup C}| = t-2k$.  By property (3), the
  rank-$2k$ cyclic flat $D$ can have nonempty intersection with at
  most two of $A$, $B$, and $C$, and such intersections have size $k$
  by property (1).  That would give $|D|\leq 2k+t-2k=t$, contrary to
  the size of rank-$2k$ cyclic flats.  Thus, property (7) holds.

  \begin{itemize}
  \item[(8)] \emph{There are no rank-$2k$ cyclic flats besides $A$,
      $B$, $C$, and $D$.}
  \end{itemize}
  
  Assume to the contrary that $N$ has at least five rank-$2k$ cyclic
  flats, say $A$, $B$, $C$, $D$, and $X$.  By applying property (5) to
  $\{X,A,B,C\}$, we may assume that $X\cap A=\emptyset$.  By applying
  property (5) to $\{X,B,C,D\}$, we may assume that
  $X\cap D=\emptyset$.  Property (7) now gives $A\cap D\ne\emptyset$,
  so $|A\cap D|=k$ by property (1).  Therefore
  $|A\cup D\cup X| = 5k+3t$.  By applying property (5) to
  $\{A,B,C,D\}$ and knowing that $A\cap D\ne\emptyset$, we may assume
  that $A\cap C=\emptyset$.  However, since $C$ contains at most $k$
  elements in each of $D$ and $X$, we get
  $|A\cup D\cup X\cup C|\geq 5k+4t$, which is impossible since
  $|E(N)|=4k+4t$.

  By properties (3) and (8), there are at most eight flags
  $F\subset G$ of cyclic flats for which $r(F)=2k$ and $r(G)=3k$.  Now
  assume that $N$ has eight such flags.  Thus, each of the four
  rank-$2k$ cyclic flats is a subset of two rank-$3k$ cyclic flats.
  The corresponding conclusion in $N^*$, translated back to $N$, shows
  that there are four rank-$3k$ cyclic flats, each of which has two
  subsets that are rank-$2k$ cyclic flats.  From property (5), we may
  assume that $A\cap C=\emptyset$.  Therefore $A\cap B$, $A\cap D$,
  $C\cap B$, and $C\cap D$ are nonempty, so property (5) gives
  $B\cap D=\emptyset$.  Thus, the $k$-element sets $A\cap B$,
  $A\cap D$, $C\cap B$, and $C\cap D$ are pairwise disjoint.  With
  this, it follows that there is a bijection $\phi:E(M)\to E(N)$ that
  induces a rank-preserving bijection from $\mathcal{Z}(M)$ onto
  $\mathcal{Z}(N)$, and so $N$ is isomorphic to $M$.

  The conclusion that $M$ is $\mathcal{G}$-unique follows since all of
  the information used above is determined by valuative invariants.
\end{proof}  

In the next example, we construct a large class of fundamental
transversal semi-designs and then focus on a particular subclass that
we show to consist of extremal matroids that satisfy hypothesis (ii)
in Theorem \ref{thm:frpbeyondPMD}.  To take advantage of geometric
insight, we define these matroids using Brylawski's geometric
description of transversal matroids, which we reviewed in Section
\ref{subsec:petran}.  One can easily write down a presentation from
this.

\begin{example}\label{ex:rank4bicircular}
  Fix $r\in\mathbb{N}$, a nonempty subset $Q$ of $[r]$, and
  $a_i\in\mathbb{N}$ for each $i\in Q$.  Pick pairwise disjoint sets
  $B=\{b_1,b_2,\ldots,b_r\}$ and $A_T$, for each $T\subseteq [r]$ for
  which $|T|\in Q$, where $|A_T|=a_{|T|}$.  Start with the uniform
  matroid $U_{r,r}$ on $B$.  For each $T\subseteq[r]$ for which
  $|T|\in Q$, iterate the operation of principal extension to add all
  elements of $A_T$ to the flat spanned by $\{b_i\,:\,i\in T\}$.  Let
  the resulting matroid be $M$.  Note that $M$ is a fundamental
  transversal matroid: $B$ is a basis, and each cyclic flat of $M$ is,
  by construction, the closure of a subset of $B$.  To see that $M$ is
  a semi-design, note that if $h$ is the least element of $Q$, then
  $R_{\mathcal{Z}(M)} =\{0,h,h+1,\ldots,r\}$, and for
  $j\in R_{\mathcal{Z}(M)}-\{0\}$, the size of all rank-$j$ cyclic
  flats of $M$ is
  $$j+\sum_{i\in Q\,:\, i\leq j}\binom{j}{i}a_i,$$
  so all cyclic flats of the same rank have the same size, as property
  (SD1) requires.  Property (SD2) holds since
  $\binom{j+1}{i}>\binom{j}{i}$ for all $i,j\in\mathbb{N}$.

  The particular type of fundamental transversal semi-design that in
  Theorem \ref{thm:bic} we show to be extremal and to satisfy
  hypothesis (ii) in Theorem \ref{thm:frpbeyondPMD} is the bicircular
  matroid that we denote by $B_{r,t}$, which is the rank-$r$ matroid
  that the construction above produces when $Q=\{2\}$ and
  $a_2=t \in\mathbb{N}-\{1\}$.  Two bicircular matroids $B_{r,t}$ are
  shown in Figure \ref{fig:rank4bicircular}.  This is the bicircular
  matroid that one obtains from the graph on $r$ vertices that has one
  loop at each vertex and $t$ parallel edges between each pair of
  vertices. (See, e.g., \cite{oxley,bic} for general bicircular
  matroids.) The bicircular matroid $B_{r,t}$ has cyclic flats of each
  rank $j\in[0,r]-\{1\}$, and the rank-$j$ cyclic flats have
  $j+\binom{j}{2}t$ elements.
\end{example}

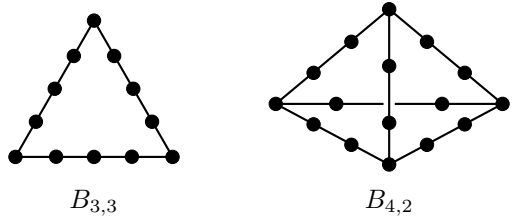
\begin{figure}
  \centering
    \begin{tikzpicture}[scale=1]
    \draw[thick](90:1.2)--(210:1.2)--(330:1.2)--(90:1.2);%

    \filldraw (90:1.2) circle (2.5pt);%
    \filldraw (210:1.2) circle (2.5pt);%
    \filldraw (330:1.2) circle (2.5pt);%

    \filldraw (350:0.78) circle (2.5pt);%
    \filldraw (30:0.6) circle (2.5pt);%
    \filldraw (70:0.78) circle (2.5pt);%

    \filldraw (110:0.78) circle (2.5pt);%
    \filldraw (150:0.6) circle (2.5pt);%
    \filldraw (190:0.78) circle (2.5pt);%

    \filldraw (230:0.78) circle (2.5pt);%
    \filldraw (270:0.6) circle (2.5pt);%
    \filldraw (310:0.78) circle (2.5pt);%

    \node at (0,-1.2) {$B_{3,3}$};%
  \end{tikzpicture}
  \hspace{1cm}
  \begin{tikzpicture}[scale=1]
    \node[inner sep = 0mm] (1) at (0,1) {};%
    \node[inner sep = 0mm] (1') at (1.45,1) {};%
    \node[inner sep = 0mm] (2) at (1.5,0.2) {};%
    \node[inner sep = 0mm] (3) at (3,1) {};%
    \node[inner sep = 0mm] (3') at (1.55,1) {};%
    \node[inner sep = 0mm] (4) at (1.5,2.25) {};%
    \foreach \from/\to in {1/2,1/1',3'/3,1/4,2/3,2/4,3/4}
    \draw[thick](\from)--(\to);%

    \filldraw (0,1) circle (2.5pt);%
    \filldraw (3,1) circle (2.5pt);%
    \filldraw (1.5,0.2) circle (2.5pt);%
    \filldraw (1.5,2.25) circle (2.5pt);%
    \filldraw (1.5,1.5) circle (2.5pt);%
    \filldraw (1.5,0.75) circle (2.5pt);%
    \filldraw (0.8,1) circle (2.5pt);%
    \filldraw (2.2,1) circle (2.5pt);%
    \filldraw (2,0.46) circle (2.5pt);%
    \filldraw (2.5,0.73) circle (2.5pt);%
    \filldraw (1,0.46) circle (2.5pt);%
    \filldraw (0.5,0.73) circle (2.5pt);%
    \filldraw (0.5,1.41) circle (2.5pt);%
    \filldraw (1,1.82) circle (2.5pt);%
    \filldraw (2.5,1.41) circle (2.5pt);%
    \filldraw (2,1.82) circle (2.5pt);%

    \node at (1.5,-0.3) {$B_{4,2}$};%
  \end{tikzpicture}
  \caption{The bicircular matroids $B_{3,3}$ and $B_{4,2}$ in Example
    \ref{ex:rank4bicircular}.}
  \label{fig:rank4bicircular}
\end{figure}

The next result is closely related to \cite[Theorems 6.2 and
6.3]{maxvalinv}, which show that Dowling matroids are extremal.  The
proofs share some common elements.

\begin{thm}\label{thm:bic}
  Fix integers $r\geq 3$ and $t\geq 2$.  Let $M$ be a simple rank-$r$
  semi-design on $r+\binom{r}{2}t$ elements for which, for each
  $h\in[r]-\{1\}$, rank-$h$ cyclic flats have $h+\binom{h}{2}t$
  elements.  Then $M$ has at most $r$ cyclic hyperplanes.
  Furthermore, the statements below are equivalent:
  \begin{itemize}
  \item[(i)] $M$ has $r$ cyclic hyperplanes,
  \item[(ii)] among rank-$r$ semi-designs $N$ on $r+\binom{r}{2}t$
    elements for which rank-$i$ cyclic flats have $i+\binom{i}{2}t$
    elements, $M$ maximizes the valuative invariant $F(N;S)$, where
    $S=S_{B_{r,t}}=\{ h+\binom{h}{2}t\,:\,h\in[0,r]\}$, and
  \item[(iii)] $M$ is isomorphic to $B_{r,t}$.
  \end{itemize}
  Thus, $B_{r,t}$ is extremal and $\mathcal{G}$-unique. 
\end{thm}

\begin{proof}
  By Lemma \ref{lem:msubiext}, for any rank-$h$ flat $F$ of $M$ that
  is not cyclic, $|F|\leq h+\binom{h-1}{2}t$.  Note that for any
  cyclic flat $F$ of $M$ of rank at least three, the restriction $M|F$
  satisfies the same hypothesis, using $r(F)$ in place of $r$.

  We first prove properties (1) and (2) below by induction on $r$.
  Property (1) is the bound on the number of cyclic hyperplanes in the
  theorem and property (2) provides the tools that we will use to
  prove the equivalence of statements (i)--(iii) in the theorem.
  \begin{itemize}
  \item[(1)] The matroid $M$ has at most $r$ cyclic hyperplanes.
  \item[(2)] If $M$ has $r$ cyclic hyperplanes $H_1,H_2,\ldots,H_r$
    and we let $F_I=\bigcap_{i\in I}H_i$ for each $I\subseteq[r]$,
    then
    \begin{itemize}
    \item[(2a)] $r(F_I)=r-|I|$, and if $I,J\subseteq[r]$ and $I\ne J$,
      then $F_I\ne F_J$,
    \item[(2b)] if $|I|\ne r-1$, then the flat $F_I$ is cyclic,
    \item[(2c)] the set $\{b_1,b_2,\ldots,b_r\}$, where
      $\bigcap_{k\in[r]-\{i\}}H_k=\{b_i\}$ for each $i\in[r]$, is a
      basis of $M$ and $E(M)$ is the union of the $\binom{r}{2}$ lines
      $\cl(\{b_i,b_j\})=F_{[r]-\{i,j\}}$, for $\{i,j\}\subsetneq [r]$,
      and
    \item[(2d)] the only nonempty cyclic flats of $M$ are the sets
      $F_I$, for $|I|\leq r-2$.
    \end{itemize}
  \end{itemize}
  In addition to the base case $r=3$, we treat $r=4$ separately since
  an inequality that we use for the induction step applies only for
  $r>4$.  Let $H_1,H_2,\ldots,H_u$ be the cyclic hyperplanes of $M$.
  
  First consider $r=3$, so $M$ has $3+3t$ elements and the cyclic
  hyperplanes $H_i$ are $(2+t)$-point lines.  We may assume that
  $u\geq 3$. If $H_i\cap H_j=\emptyset$ for some
  $\{i,j\}\subsetneq [3]$, then we would have
  $$|H_1\cup H_2\cup H_3|\geq 2(2+t)+t=4+3t,$$ which contradicts
  having $|E(M)|=3+3t$.  If $H_1\cap H_2\cap H_3\ne\emptyset$, then we
  would have
  $$|H_1\cup H_2\cup H_3|\geq 1+ 3(1+t)=4+3t,$$ so this too is
  impossible.  Thus, $|H_i\cap H_j|=1$, for all
  $\{i,j\}\subsetneq [3]$, and $H_1\cap H_2\cap H_3=\emptyset$, from
  which we get $|H_1\cup H_2\cup H_3|=3+3t$, and so
  $H_1\cup H_2\cup H_3=E(M)$.  Thus, any cyclic line contains at least
  two points in one of $H_1,H_2,H_3$ and so is that cyclic line, so
  $H_1,H_2,H_3$ are the only cyclic lines, which proves property (1).
  Also, we have checked all parts of property (2).
  
  Consider $r=4$, so $M$ has $4+6t$ elements and the cyclic
  hyperplanes $H_1,H_2,\ldots,H_u$ are $(3+3t)$-point planes.  Since
  $$|H_1\cap H_2|\geq |H_1|+|H_2|-|E(M)|=2,$$ the line
  $H_1\cap H_2$ has either two or $2+t$ elements.  First assume that
  $|H_1\cap H_2|= 2$.  Then $H_1\cup H_2=E(M)$, so for any other plane
  $P$, since $|P\cap H_i|\leq 2+t$ for $i\in[2]$, we have
  $|P|\leq 2(2+t)<3+3t$.  Thus, $H_1$ and $H_2$ would be the only
  cyclic planes.  We may now assume that the intersection of any two
  cyclic planes of $M$ is a cyclic line.  No cyclic line is in three
  cyclic planes, for otherwise the union of the three cyclic planes
  would contain $2+t+3(1+2t)=5+7t$ elements, but $|E(M)|=4+6t$.  Since
  any cyclic plane contains at most three cyclic lines, each of which
  is in at most one more cyclic plane, there are at most four cyclic
  planes, so property (1) holds when $r=4$.  Now assume that
  $H_1,H_2,H_3,H_4$ are the cyclic planes of $M$.  Using the notation
  $F_I$ defined in property (2), the argument above shows that
  $F_{\{i,j\}}$ is a cyclic line for all $\{i,j\}\subsetneq [4]$ and
  no two such lines are equal, so each $H_i$ contains three cyclic
  lines.  From property (2) for the restrictions $M|H_i$ (the case
  $r=3$), we now get that $F_{[4]}=\emptyset$ and that $F_{[4]-\{i\}}$
  is a singleton $\{b_i\}$ for all $i\in [4]$.  From property (2) for
  $M|H_1$, the set $\{b_2,b_3,b_4\}$ is a basis of $M|H_1$; thus
  $\{b_1,b_2,b_3,b_4\}$ is a basis of $M$ since $b_1\not\in H_1$.  The
  six cyclic lines $\cl(\{b_i,b_j\})$ are the intersections of pairs
  of cyclic planes.  Counting shows that the union of the six cyclic
  lines $\cl(\{b_i,b_j\})$, for $\{i,j\}\subsetneq [4]$, is $E(M)$.
  Finally, let $L$ be any cyclic line. If $b_i\in L$, then since the
  union of the cyclic planes that contain $b_i$ is $E(M)$, the line
  $L$ must be in one of these planes and hence be $\cl(\{b_i,b_j\})$
  for some $j\in[4]-\{i\}$.  Now assume that
  $L\cap \{b_1,b_2,b_3,b_4\}=\emptyset$.  Since $E(M)$ is the union of
  the six cyclic lines $\cl(\{b_i,b_j\})$, we may assume that
  $e\in L\cap \cl(\{b_i,b_j\})$.  Thus, $\cl(L\cup \{b_i\})$ is a
  cyclic plane not among $H_j$ for $j\in [4]$, contrary to property
  (1).  Thus, property (2) holds when $r=4$.

  Now assume that $r>4$ and that properties (1) and (2) hold for lower
  ranks.  We first show that if $ \{i,j\}\subseteq [u]$, then
  $F_{\{i,j\}}$ is a rank-$(r-2)$ cyclic flat.  Now
  $|H_i\cup H_j|\leq r+\binom{r}{2}t$, so
  \begin{align*}
    |F_{\{i,j\}}|
    &\,= |H_i|+|H_j|-|H_i\cup H_j|\\
    &\,\geq 2\Bigl(r-1+\binom{r-1}{2}t\Bigr)- \Bigr(
      r+\binom{r}{2}t\Bigl)\\
    &\,= r-2+\Bigl( 2\binom{r-1}{2}-\binom{r}{2}\Bigl)t \\
    &\,>  r-2+ \binom{r-3}{2}t
  \end{align*}
  (the final inequality requires $r>4$).  It follows that
  $F_{\{i,j\}}$ is a rank-$(r-2)$ cyclic flat since if
  $r(F_{\{i,j\}}) <r-2$, then
  $|F_{\{i,j\}}|\leq r-3 +\binom{r-3}{2}t$, while if
  $r(F_{\{i,j\}})=r-2$ but $F_{\{i,j\}}$ were not cyclic, then
  $|F_{\{i,j\}}|\leq r-2+\binom{r-3}{2}t$.

  The differences $H_i-F_{\{i,j\}}$ and $H_j-F_{\{i,j\}}$ each
  contain $$r-1+\binom{r-1}{2}t-(r-2)-\binom{r-2}{2}t =1+(r-2)t$$ of
  the $$r+\binom{r}{2}t -(r-2)-\binom{r-2}{2}t =2+(2r-3)t$$ elements
  in $E(M)-F_{\{i,j\}}$, so $H_i$ and $H_j$ are the only cyclic
  hyperplanes that contain $F_{\{i,j\}}$.  Thus, all $\binom{u}{2}$
  flats $F_{\{i,j\}}$, with $\{i,j\}\subseteq[u]$, are different.  In
  particular, $F_{\{i,u\}}$, for $i\in[u-1]$, are $u-1$ different
  cyclic hyperplanes of $M|H_u$, so the induction assumption applied
  to $M|H_u$ gives $u-1 \leq r-1$, so $u\leq r$, so property (1)
  holds.

  Now assume that $M$ has $r$ cyclic hyperplanes,
  $H_1,H_2,\ldots,H_r$.  The argument above shows that for each
  $i\in[r]$, the restriction $M|H_i$ has $r-1$ cyclic hyperplanes,
  namely, $F_{\{i,j\}}$ for each $j\in[r]-\{i\}$, so property (2)
  holds for $M|H_i$.  Properties (2a) and (2b) immediately follow for
  $M$. Property (2c) follows since $\{b_1,b_2,\ldots,b_{r-1}\}$ is a
  basis of $M|H_r$ by the induction hypothesis, and $b_r\not\in H_r$.
  To prove property (2d), if there were a cyclic flat that is not of
  the form $F_I$, then let $F$ be such a cyclic flat of greatest rank.
  Since $H_1,H_2,\ldots,H_r$ are the only cyclic hyperplanes in $M$,
  we have $r(F)\leq r-2$.  Since $F$ is not of the form $F_I$, there
  is an $e\in F-\{b_1,b_2,\ldots,b_r\}$, say with
  $e\in \cl_M(\{b_i,b_j\})$, for which
  $\cl_M(\{b_i,b_j\})\not\subseteq F$.  By the maximality assumption,
  the rank$(r(F)+1)$ cyclic flat $\cl_M(F\cup \{b_i\})$ is $F_K$ for
  some $K\subsetneq [r]$.  But the hyperplane $F$ of $M|F_K$
  contradicts property (1) for $M|F_K$.  Thus, property (2) holds.
  
  We now prove the equivalence of statements (i)--(iii).  Recall that
  $S=S_{B_{r,t}}$, that is, $S=\{ h+\binom{h}{2}t\,:\,h\in[0,r]\}$.
  Since $M$ has at most $r$ cyclic hyperplanes and the restriction of
  $M$ to a cyclic hyperplane satisfies the hypotheses of the theorem,
  with $r-1$ in place of $r$, it follows by induction that
  $F(M;S)\leq r!\cdot (t+2)/2$, and if $F(M;S)= r!\cdot (t+2)/2$, then
  $M$ has $r$ cyclic hyperplanes.  The equivalence of statements (i)
  and (ii) follows.  Statement (iii) implies statement (i) since
  $B_{r,t}$ has $r$ cyclic hyperplanes.  Now assume that statement (i)
  holds.  By what we showed in property (2), it follows that the
  cyclic flats of $M$ can be identified with those of $B_{r,t}$ by a
  bijection from $E(M)$ onto $E(B_{r,t})$ that preserves rank, so $M$
  is isomorphic to $B_{r,t}$.

  Since the invariants that are used to impose the conditions above
  are valuative and we deduced that any matroid that meets them is
  isomorphic to the bicircular matroid $B_{r,t}$, it follows that
  $B_{r,t}$ is extremal and $\mathcal{G}$-unique.
\end{proof}

\begin{center}
 \textsc{Acknowledgments}
\end{center}

\vspace{3pt}

The author is grateful to Christos Athanasiadis and Joel Lewis for
helpful comments that drew \cite{MK} to his attention, and to Alex
Fink for insights into Lemma \ref{lem:dproduct} via Euler
characteristics of chain complexes.

\end{document}